\documentclass[12pt]{amsart}
\usepackage{amsmath,amssymb,amsthm, graphicx}

\def\f{{\mathcal{F}}}

\def\p{{\mathfrak{p}}}

\def\Sym{{\rm Sym}}
\def\C{\mathbb{C}}
\def\Z{\mathbb{Z}}
\def\Q{\mathbb{Q}}

\def\F{\mathbb{F}}

\def\ord{{\rm ord}}

\newcommand{\inflim}[1]{\displaystyle \lim_{#1 \rightarrow \infty}}

\def\Z{{\mathbb Z}}

\newcommand\Fq{\mathbb{F}_q}

\newenvironment{Gal}{{\rm Gal\,}}{}

\newenvironment{Aut}{{\rm Aut}}{}
\newenvironment{Frob}{{\rm Frob}}{}

\newenvironment{Fix}{{\rm Fix}}{}
\newenvironment{img}{{\rm IMG}}{}
\newenvironment{imgbar}{\overline{\rm IMG}}{}
\newenvironment{critf}{\textrm{Crit} \, f}{}

\newtheorem{theorem}{Theorem}[section]
\newtheorem{definition}[theorem]{Definition}
\newtheorem{lemma}[theorem]{Lemma}
\newtheorem{proposition}[theorem]{Proposition}
\newtheorem{proposition-definition}[theorem]{Proposition-Definition}
\newtheorem{corollary}[theorem]{Corollary}

\theoremstyle{definition}
\newtheorem{example}[theorem]{Example}

\theoremstyle{remark}
\newtheorem*{remark}{Remark}

\begin{document}
\title[Fixed-point-free elements]{Fixed-point-free elements of iterated monodromy groups}
\author{Rafe Jones}
\date{\today}
\thanks{The author's research was partially supported by NSF grant DMS-0852826.}

\maketitle

\begin{abstract}
The iterated monodromy group of a post-critically finite complex polynomial of degree $d \geq 2$ acts naturally on the complete $d$-ary rooted tree $T$ of preimages of a generic point.  This group, as well as its pro-finite completion, act on the boundary of $T$, which is given by extending the branches to their ``ends" at infinity. We show that in most cases, elements that have fixed points on the boundary are rare, in that they belong to a set of Haar measure $0$. The exceptions are those polynomials linearly conjugate to multiples of Chebyshev polynomials and a case that remains unresolved, where the polynomial has a non-critical fixed point with many critical pre-images.  The proof involves a study of the finite automaton giving generators of the iterated monodromy group, and an application of a martingale convergence theorem.  Our result is motivated in part by applications to arithmetic dynamics, where iterated monodromy groups furnish the ``geometric part" of certain Galois extensions encoding information about densities of dynamically interesting sets of prime ideals.  
\end{abstract}

\section{Introduction}

Suppose that $f \in \C[x]$ is a polynomial of degree $d \geq 2$, and let $\critf$ be the set of critical points of $f$.   Define the {\em post-critical set} $P_f := \bigcup_{n \geq 1} f^n(\critf)$, where $f^n$ denotes the $n$th iterate of $f$.  Note that $P_f$ consists of all points over which the map $f^n : \C \to \C$ is ramified for at least one $n$.  When $P_f$ is finite, we call $f$ {\em post-critically finite}.  Choose a point $\beta \in \C \setminus P_f$, and denote by $T_\beta$ the set of all preimages of $\beta$ under some iterate of $f$.  Then $T_\beta$ is a complete rooted $d$-ary tree whose $n$th level is given by 
$f^{-n}(\beta)$.   The fundamental group $\pi_1(\C \setminus P_f)$ acts on each set $f^{-n}(\beta)$ by monodromy, and thus gives a subgroup of $\Aut(T_\beta)$ that we call the \textit{iterated monodromy group} of $f$, and write $\img(f)$.   

In the twenty years since their introduction, iterated monodromy groups have become a powerful tool used in a variety of settings.  They are both computable and have deep connections to the dynamics of the underlying polynomial.  Indeed, the action of a set of generators of $\img(f)$ on $T_\beta$ can be given by a simple finite automaton that depends largely on the structure of the set $P_f$ (see Section \ref{background2} for more details).  To illustrate the connections to dynamics, one can associate to $\img(f)$ a \textit{limit dynamical system} whose points are equivalence classes of left-infinite paths in $T_\beta$, and whose map is the shift map. This dynamical system is topologically conjugate to the action of $f$ on its Julia set \cite[Section 3.6 and Theorem 6.4.4]{nek}. 
For this reason the group $\img(z^2 - 1)$ has become known as the \textit{Basilica group}, \label{basilica} since the top half of its Julia set bears a striking resemblance to the profile of the Basilica di San Marco in Venice.  

Applications of iterated monodromy groups abound.  Defined in a more general setting, they have been used by Bartholdi and Nekrashevych to resolve the well-known ``twisted rabbit" problem of J. Hubbard \cite{bartnek}.  They have also attracted interest for their purely group-theoretic properties; for instance, the Basilica group is the first known example separating the classes of amenable groups and sub-exponentially amenable groups \cite{bartvir}.  The more general class of groups generated by finite automata includes the renowned Grigorchuk group \cite{grigorchuk}, the first example of a group of intermediate growth.  The monograph \cite{nek} gives an overview of iterated monodromy groups and their applications, as well as an extensive bibliography.  





Our interest in iterated monodromy groups comes from arithmetic, where properties of the action of $\img(f)$ on the boundary of $T_\beta$ yield information about interesting sets of prime ideals (See Section \ref{connections1} for more details). In particular, we are interested in elements of $\img(f)$ that fix at least one point on the boundary of $T_\beta$, or equivalently, at least one infinite branch of $T_\beta$.  Our main result is that such elements are rare for a large class of $f$.  

Let us introduce some notation.  We may identify $T_\beta$ with the set $X^*$ of all finite words (including the empty word) over an alphabet $X$ containing $d$ letters.  The root of $T_\beta$ corresponds to the empty word, and $f^{-n}(\beta)$ corresponds to $X^n$, the set of all words of length $n$.  The boundary of $X^*$ is the set $X^\omega = \{x_1x_2x_3 \cdots : x_i \in X\}$ of {\em ends} of $X^*$.  Let $G$ be a group of automorphisms of $X^*$, and denote by $G_n$ the image of $G$ under the restriction map $\Aut(X^*) \to \Aut(X^n)$.   Define
\begin{equation} \label{fdef}
\f(G) = \lim_{n \to \infty} \frac{\#\{g \in G_n : \text{$g$ fixes at least one element of $X^n$}\}}{\#G_n}.
\end{equation}
Note that the fraction above is non-increasing, since all lifts to $G_{n+1}$ of an element of $G_n$ with no fixed points again have no fixed points.  Thus the limit in \eqref{fdef} exists.  We may also describe $\f(G)$ by considering the closure $G_\infty$ of $G$ in $\Aut(X^*)$, which is a compact topological group and thus comes with a natural probability measure $\mu$ (the normalized Haar measure).  It is straightforward to show that 
\begin{equation} \label{fdef2} 
\f(G) = \mu(\{g \in G_\infty : \text{$g$ fixes at least one element of $X^\omega$}\}).
\end{equation}
Note that $\f(G)$ is determined by $G$ rather than $G_\infty$, even though we have made reference to $G_\infty$ in \eqref{fdef2}.  We use the notation $G_\infty$ for the closure of $G$ because the latter coincides with the inverse limit of the groups $G_n$.

Following the terminology of \cite[section 1.3]{smirnov}, define $f \in \C[z]$ to be \textit{exceptional} if there exists a finite, non-empty set $\Sigma$ with $f^{-1}(\Sigma) \setminus \critf = \Sigma$. 
Our main result is the following.
\begin{theorem} \label{monodromy}
Let $f \in \mathbb{C}[z]$ be a post-critically finite polynomial of degree at least two, with iterated monodromy group $G$. If $f$ is not exceptional, then $\mathcal{F}(G) = 0$.  
\end{theorem} 


Exceptional polynomials have appeared as a distinguished class in a variety of settings; for instance, the affine orbifold lamination attached to certain exceptional polynomials has an isolated leaf \cite[Section 2]{lkr} (see also\cite{lyubich, earlysmirnov, smirnov} for special properties of such polynomials).  It is not difficult to show that if $f$ is exceptional, then $\#\Sigma \leq 2$ (see p. \pageref{exceptdisc}). If $\#\Sigma = 2$, then $f$ is linearly conjugate to $\pm T_d$, where $T_d$ is the Chebyshev polynomial of degree $d$ (see Proposition \ref{chebclass}). Note that linear conjugacy preserves the conjugacy class in $\Aut(X^*)$ of $\img(f)$.  If $\#\Sigma = 1$, then $f$ has a fixed point $z_0$ all of whose preimages are critical except $z_0$ itself, so $f$ is conjugate to a polynomial of the form
\begin{equation} \label{oneptexcept}
z(z-a_1)^{k_1} \cdots (z - a_m)^{k_m},
\end{equation}
where $a_i \in \C \setminus \{0\}$ and $k_i \geq 2$. 

We also compute $\f(G)$ for exceptional polynomials with $\#\Sigma = 2$. Note that $T_d$ is conjugate to $-T_d$ when $d$ is even. 
\begin{proposition} \label{exceptprop}
If $f(z) \in \C[z]$ is conjugate to $T_d$ for $d$ even, then $\f(\img(f)) = 1/4$. If $f$ is conjugate to $\pm T_d$ for $d$ odd, then $\f(\img(f)) = 1/2$. 
\end{proposition}
Thus the only post-critically finite polynomials $f$ for which $\f(\img(f))$ remains unknown are non-Chebyshev maps conjugate to a map of the form \eqref{oneptexcept}. We remark that the power maps $f(z) = z^d$ are not exceptional, and hence have $\f(\img(f)) = 0$, in contrast to Chebyshev polynomials. 


When $f$ is quadratic, it must be conjugate to $z^2 + c$ for $c \in \C$. The only exceptional polynomial of this form is $f(z) = z^2 - 2$, since those of the form \eqref{oneptexcept} have degree at least 3.  Theorem \ref{monodromy} and Proposition \ref{exceptprop} thus give the result that furnished the original motivation for this project:
\begin{corollary} \label{quadpoly}
Let $f(z) = z^2 + c$ be post-critically finite, and let $G$ be its iterated monodromy group.  Then $\mathcal{F}(G) = 0$ unless $f(z)$ is the Chebyshev polynomial $z^2 - 2$, in which case $\mathcal{F}(G) = 1/4$.    
\end{corollary}

To prove Theorem \ref{monodromy}, we study groups of automorphisms of rooted trees, and draw heavily on a characterization due to V. Nekrashevych \cite[Theorem 6.10.8]{nek} of which such groups are $\img(f)$ for some post-critically finite $f$.  Along the way we derive some results that apply more generally.  For instance, define $g \in \Aut(X^*)$ to be \textit{spherically transitive} if it acts transitively on $X^n$ for each $n \geq 1$.  Every iterated monodromy group of a polynomial contains a spherically transitive element (see Theorem \ref{char} and Lemma \ref{trans}), which is furnished by monodromy at infinity.  This element plays crucial role in our analysis.  

\begin{theorem} \label{gen}
Suppose that $G \leq \Aut(X^*)$ has a spherically transitive element.  Then $\mu(\{g \in G_\infty : \text{$g$ fixes infinitely many elements of $X^\omega$}\}) = 0$.
\end{theorem}

We prove Theorem \ref{gen} in Section \ref{fpprocess}, where we define a stochastic process encoding information about fixed-point-free elements of $G_n$.  The presence of a spherically transitive element implies this process is a martingale (Theorem \ref{mart}), and we establish Theorem \ref{gen} using a basic martingale convergence theorem. 

We give two other results that lead up to the proof of Theorem \ref{monodromy}.  A salient feature of $X^*$ is its self-similarity, and we use this to describe elements of $\Aut(X^*)$ recursively.  Let $g \in \Aut(X^*)$, and for a vertex $v \in X^*$ consider the subtrees $vX^*$ and $g(v)X^*$ with root $v$ and $g(v)$, respectively.  Both are naturally isomorphic to $X^*$, and identifying them gives an automorphism $g|_v \in \Aut(X^*)$, called the {\em restriction} of $g$ at $v$.  
See Section \ref{background1} for examples and further definitions.  
We call $G$ {\em contracting} if there is a finite set $\mathcal{N} \subset G$ such that for each $g \in G$, there is $n_g \geq 1$ such that all restrictions of $g$ at words of length at least $n_g$ belong to $\mathcal{N}$.  Roughly, this property means that the action of $g$ is relatively restrained, at least close to the boundary of $X^*$.  In particular, many computations in $G$ can be reduced to finite considerations; see Section \ref{background1} for more details.  It is known that iterated monodromy groups of post-critically finite polynomials are contracting \cite[Theorems 3.9.12 and 6.10.8]{nek}. 
Let \label{ndefs}
\begin{align*} 
\mathcal{N}_1 & = \{g \in G : \text{$g|_v = g$ and $g(v) = v$ for some non-empty $v \in X^*$}\}.
\end{align*}
When $G$ is contracting, $\mathcal{N}_1$ is finite; see Proposition \ref{contracting}.  
\begin{theorem} \label{crystal}
Suppose that $G \leq \Aut(X^*)$ is contracting and has a spherically transitive element.  
If every $g \in \mathcal{N}_1$ fixes infinitely many ends of $X^*$, then $\mathcal{F}(G) = 0$.  
\end{theorem}

It is not hard to show that when $G$ is contracting, $\mathcal{N}_1$ is torsion (see the end of Section \ref{fpprocess}), and this gives
\begin{corollary} \label{torfree}
Suppose that $G \leq \Aut(X^*)$ is contracting and has a spherically transitive element.  If $G$ is torsion-free, then $\mathcal{F}(G) = 0$.  
\end{corollary}

The Basilica group $B$ is known to be torsion-free \cite{bartnek}, and so Corollary \ref{torfree} proves that $\f(B) = 0$.  

Theorem \ref{crystal} and Corollary \ref{torfree} are proven using only the tools from Section \ref{fpprocess}, which do not use specific facts about iterated monodromy groups.  
On the other hand, in order to use Theorem \ref{crystal} to prove Theorem \ref{monodromy}, we apply a characterization of iterated monodromy groups of post-critically finite polynomials due to Nekrashevych \cite[Theorem 6.10.8]{nek} (we give a restatement in Theorem \ref{char}).  This gives a natural finite generating set $A$ for $\img(f)$.  We introduce a \textit{kneading graph} associated to $A$, and use it to show that every element of $\mathcal{N}_1$ is conjugate to a power of an element of $\mathcal{N}_1 \cap A$ (Theorem \ref{n1}).  Thus we reduce questions about fixed points of the action of elements of $\mathcal{N}_1$ on $X^*$ to the study of the action of elements of  $A$ on $X^*$, and these are directly related to the orbits of the critical points of $f$ (Theorem \ref{compmon}).  In Section \ref{sec: last} we use a strong property of $A$ given in Theorem \ref{char} to show that only very special configurations of $A$ allow for elements of $\mathcal{N}_1$ to fix a finite number of ends of $X^*$.  

While most of the proofs here are group-theoretic, the consequences are of interest to number theorists, and hence we have made an effort to make the exposition relatively self-contained.  
After giving more details on links between our results and number theory (Section \ref{connections1}), we give in Section \ref{background1} just the background necessary to prove Theorem \ref{gen}, Theorem \ref{crystal}, and Corollary \ref{torfree}. These proofs are in Section \ref{fpprocess}. For the remainder of the paper, more background is required, which we describe in Section \ref{background2}.  Sections \ref{sec: treelike}, \ref{sec: kneadgraph}, and \ref{sec: last} contain the rest of the proofs. 

\section{Connections to number theory and arithmetic dynamics} \label{connections1}

We give here some links between iterated monodromy groups and density questions for sets of dynamical interest in arithmetic contexts. Work is ongoing to exploit these connections to produce density results. 

\subsection{Iterated monodromy groups and Galois groups} \label{mono1}

Let $f \in \C[z]$ be post-critically finite. The action of $\imgbar(f)$ on the set $f^{-n}(\beta)$, where $\beta$ is outside the post-critical set, is given by monodromy, and we refer to this quotient as $G_n$. On the other hand, the Galois group $\tilde{G}_n := \Gal(K_n/\C(t))$, where $K_n$ is the splitting field of the polynomial $f^n(z) - t \in \C(t)[z]$, has a natural action on the $d^n$ roots of $f^n(z) - t$ over $\C(t)$.   It is well-known (see e.g. \cite[Theorem 8.12]{forster}) that $G_n \cong \tilde{G}_n$, and the corresponding actions on $f^{-n}(\beta) \subset \C$ and $f^{-n}(t) \subset \overline{\C(t)}$ are conjugate subgroups of $S_{d^n}$.  We summarize this in the following proposition, which is essentially {\cite[Proposition 6.4.2]{nek}}.
\begin{proposition} \label{cxgal}
The profinite iterated mondromy group $\overline{\img}(f)$ is isomorphic to the Galois group of $K_\infty$ over $\C(t)$, where $K_\infty = \bigcup_{n \geq 1} K_n$.  Moreover, the corresponding actions on the preimage trees $T_\beta \subset \C$ and $T_t \subset \overline{\C(t)}$ are conjguate. 
\end{proposition}
Proposition \ref{cxgal} prompted the introduction of iterated monodromy groups, as a tool for computing the group $\tilde{G}_\infty$ \cite[p.174]{nek}. In the remainder of this section, we give connections of iterated monodromy groups to arithmetic probelms, which proceed via the link to Galois theory in Proposition \ref{cxgal}. 

\subsection{Density problems: global fields}

Let $k$ be a perfect field of characteristic $p \geq 0$, let $\overline{k}$ be an algebraically closed field containing $k$, and suppose that $f \in k[x]$ has degree $d \geq 2$ that is prime to $p$.  For the moment we do not assume that $f$ is post-critically finite.  
Let $K_n$ be the splitting field of $f^n(x) - t$ over $k(t)$, and put $K_\infty = \bigcup_n K_n$.  For $n \leq \infty$, define the {\em arithmetic monodromy group} $A_n$ to be the Galois group of $K_n$ over $k(t)$.  The field of constants of $K_n$ is $\overline{k} \cap K_n$, which we denote by $k_n'$.  The {\em geometric monodromy group} $G_n$ is the normal subgroup of $A_n$ whose elements restrict to the identity on $k'_n$.  Clearly $A_n/G_n$ is isomorphic to the Galois group of $k'_n/k$, and hence we have an exact sequence
\begin{equation} \label{es}
1 \to G_n \to A_n \to \Gal(k'_n/k) \to 1
\end{equation}    
for each $n \leq \infty$.
We may also take a specialization $t = t_0 \in k$, thereby obtaining a specialized form of \eqref{es}:
\begin{equation} \label{es2} 
1 \to G_n(t_0) \to A_n(t_0) \to \Gal(k'_n/k) \to 1.
\end{equation}    
Note that the extension $k'_n/k$ of constants is independent of specialization.   
For $n \leq \infty$, the groups $A_n$ and $G_n$ both act naturally on the set $R_n$ of roots of $f^n(x) - t$ over $k(t)$, while $A_n(t_0)$ and $G_n(t_0)$ act on the set $R_n(t_0)$ of roots of $f^n(x) - t_0$ over $k$.  We thus define
\begin{equation} \label{arithF}
\mathcal{F}(A_\infty) = \lim_{n \to \infty} \frac{\#\{g \in A_n : \text{$g$ fixes at least one root of $f^n(x) - t$}\}}{\#A_n},
\end{equation}
with similar definitions for $\mathcal{F}(G_\infty), \mathcal{F}(A_\infty(t_0)),$ and  $\mathcal{F}(G_\infty(t_0)).$  We remark that the set $\bigcup_{n \geq 1} R_n$ has a natural structure as a complete $\deg f$-ary rooted tree, while the same is true of $\bigcup_{n \geq 1} R_n(t_0)$ provided that $f^n(x) - t_0$ is separable for all $n \geq 1$ (or equivalently, there are no critical points of $f$ mapping to $t_0$ under any iterate of $f$).  We can thus identify $\bigcup_{n \geq 1} R_n$ with $X^*$, and $A_n, G_n, A_n(t_0), G_n(t_0)$ with subgroups of $\Aut(X^*)$.  With this identification, \eqref{arithF} is the same as \eqref{fdef}.  

If $f \in \C[x]$ is a post-critically finite polynomial, its coefficients must satisfy algebraic relations imposed by the self-intersections of the orbits of the critical points, and hence $f$ is defined over a finite extension $k$ of $\mathbb{Q}$.  We may take $\overline{k} = \C$, and then \eqref{es} with $n = \infty$ becomes  
\begin{equation} \label{imges} 
1 \to \img(f) \to A_\infty \to \Gal(k'_\infty/k) \to 1
\end{equation}    

Let $k$ be a global field, that is, a finite extension of $\Q$ or a finite extension of the function field $\F_q(t)$ of $\mathbb{P}^1$ over the finite field with $q$ elements, and we take the ring of integers $O_k$ to be the integral closure in $k$ of $\Q$ or $\Fq[t]$.  We wish to have a notion of size for a set of prime ideals in $O_k$.  
\begin{definition} \label{dirichlet}
Let $k$ be a global field and $\mathcal{P}$ be a set of primes in $O_k$.  The {\em Dirichlet density} of $\mathcal{P}$ is 
$$
\delta(\mathcal{P}) = \limsup_{s \rightarrow 1^{+}}  \frac{\sum_{\p \in \mathcal{P}} \;  N \p^{-s}}{\sum_{\p \subset O_k} \;  N\p^{-s}}, 
$$ 
where $N \p$ is the number of elements in the field $O_k/\p O_k$.  
\end{definition}

The Chebotarev Density theorem allows one to relate the Dirichlet density of various naturally-occuring sets of primes in $O_k$ to group-theoretic properties of the Galois groups of certain extensions of $k$.  The following theorem is an instance of this.  
\begin{theorem}[Theorem 2.1, \cite{quaddiv}] \label{numfieldden}
Let $k$ a number field with ring of integers $O_k$, and let $f \in O_k[x]$ with $a_0, t_0 \in O_k$.  Let $\mathcal{P}(f, t_0)$ be the set of primes dividing at least one element of the sequence $f^n(a_0) - t_0, n \geq 1$.  Suppose that $f^n(x) - t_0$ is separable for all $n \geq 1$, and let $A_\infty(t_0)$ be as in \eqref{es2}.  Then 
$$\delta(\mathcal{P}(f, t_0)) \leq \mathcal{F}(A_\infty(t_0)).$$
\end{theorem}

Note that the conclusion is independent of the choice of $a_0$.  In the case of number fields as above, we may also replace $\delta$ with natural density, namely
\begin{equation*}
D(\mathcal{P}) = \lim_{x \to \infty}\frac{\#\{\p \in \mathcal{P} : N\p \leq x\}}{\#\{\p : N\p \leq x\}},
\end{equation*}
and part of the conclusion of the theorem is that this limit exists.  

Theorem \ref{numfieldden} says that $\mathcal{F}(A_\infty)$ gives the ``generic" value of the density of prime divisors of an orbit of $f$ translated by a constant $t_0$.  Indeed, if
$\mathcal{F}(A_\infty) = 0$, then one can use the Hilbert irreducibility theorem to show that for any $\epsilon > 0$, $\mathcal{P}(f, t_0) < \epsilon$ for all but a thin set of $t_0$.  In the case where $f$ is post-critically finite, we have the exact sequence \eqref{imges}, and in light of Theorem \ref{monodromy}, one needs to study the extension of constant fields $k'_\infty/k$ and understand how it interacts with $\img(f)$.  Indeed, if $k'_\infty$ is a finite extension of $k$, one could replace the ground field $k$ by $k'_\infty$ and obtain the desired result.  However, it seems unlikely that this is the case in most circumstances.  For instance, when $f(x) = x^2$ and $k = \Q$, we have that $k'_\infty = \Q_{\zeta_{2^\infty}}$.  


\subsection{Density problems over finite fields}

Let $\Fq$ be the finite field with $q$ elements, let $f \in \Fq[x]$, and let $\alpha \in \overline{\mathbb{F}}_q$.  Clearly the forward orbit $\{f^n(\alpha) : n \geq 1\}$ of any such $\alpha$ is contained in a finite extension of $\Fq$, whence it must be finite.  We thus have two fundamental behaviors: if there is a $j \geq 1$ with $f^j(\alpha) = \alpha$ we call $\alpha$ {\em purely periodic} under $f$, while if there is no such $j$ then we call $\alpha$ {\em pre-periodic} under $f$.  Let ${\rm Per}(f)$ be the purely periodic points.  Note that by construction $f$ must be post-critically finite, since all its orbits are finite.  Define the Dirichlet density of a set $S \subseteq \overline{\mathbb{F}}_q$ to be 
\begin{equation}  \label{dirichdef}
\delta(S) = \limsup_{s \rightarrow 1^{+}} \frac{\sum_{\alpha \in S} \; (\deg \alpha)^{-1} N(\alpha)^{-s}}
{\sum_{\alpha \in \overline{\mathbb{F}}_q} \; (\deg \alpha)^{-1} N(\alpha)^{-s}},
\end{equation}
where $\deg \alpha = [\Fq(\alpha) : \Fq]$, and $N(\alpha) = q^{\deg \alpha}$.  This is essentially identical to Definition \ref{dirichlet}; the $(\deg \alpha)$ term is necessary because there are $\deg \alpha$ conjugates of $\alpha$ corresponding to the prime of $\Fq[t]$ with root $\alpha$. 


We sketch an argument showing how $\delta({\rm Per}(f))$ is given by statistics of an arithmetic monodromy group as in \eqref{es}, where $k = \Fq$.  Note that $\alpha \in {\rm Per}(f)$ if and only if some branch of the tree of preimages $\bigcup_{n \geq 1} f^{-n}(\alpha)$ is contained in the base field $\F_q$.  Let $\p$ be the prime ideal generated by the minimal polynomial of $\alpha$ over $\F_q$.  Then a branch of $\bigcup_{n \geq 1} f^{-n}(\alpha)$ is contained in $\F_q$ if and only if $\Frob_\p \in A_n$ fixes a root of $f^n(x) - t$ for each $n \geq 1$ (denote by $\mathcal{P}$ the set of such $\p$).  Here $\Frob_\p$ is the conjugacy class of elements of $A_n$ that act on the residue class field $O_{K_n}/\p O_{K_n}$ as $x \mapsto x^q$.   The Chebotarev density theorem for function fields \cite[Theorem 9.13A]{rosen} then gives that the Dirichlet density of $\mathcal{P}$ is bounded above, for each $n \geq 1$, by the proportion of $g \in A_n$ that fix at least one root of $f^n(x) - t$.  Thus this density is bounded above by $\mathcal{F}(A_\infty)$.  It is then straightforward to show that this implies $\delta({\rm Per}(f)) \leq \mathcal{F}(A_\infty)$.

\section{Background and examples, part I} \label{background1} 

\subsection{Wreath recursion and spherically transitive elements}  In this section we give the background required to prove the resutls in Section \ref{fpprocess}. We draw on the exposition in \cite[Chapter 1]{nek}, including following the convention there of writing group actions on the left.  From now on we suppose that our alphabet $X$ is given by $\{0, \ldots, d-1\}$, and we let $S_d$ denote the symmetric group on $d$ letters.  Then there is a natural isomorphism 
$$\psi : \Aut(X^*) \to S_d \wr \Aut(X^*),$$
where $\wr$ denotes the wreath product, that takes $g$ to $(\sigma, (g|_0, \ldots, g|_{d-1}))$, where $\sigma \in S_d$ is the action of $g$ on $X$ (i.e., on the first level of $X^*$).  In other words, we may describe $g$ by specifying its restriction at each element of $X$ and its action on $X$.  We call this the {\em wreath recursion} describing $g$.  We generally drop the parentheses and equate $g$ with its image under $\psi$, writing 
\begin{equation} \label{wreathrec}
g = \sigma(g|_0, \ldots, g|_{d-1}).
\end{equation}  
We write the identity element as $1$, and when the permutation $\sigma$ is the identity, we omit it.  Hence the identity element of $\Aut(X^*)$ is given in wreath recursion by $(1,1, \ldots, 1)$.  Note that the element $a = (a, 1, 1, \ldots, 1)$ is also the identity, since by induction it acts trivially on $X^n$ for all $n$, and thus acts trivially on $X^*$.  Given $g = \sigma(g|_0, \ldots, g|_{d-1})$, we can make explicit its action on any $X^n$ thanks to the following formulas, which are straightforward to prove: 
\begin{equation} \label{lifts}
g|_{vw} = g|_v|_w \qquad g(vw) = g(v)g|_v(w),
\end{equation}
for any $v,w \in X^*$.  

One can multiply elements in wreath recursion form using the normal multiplication in a semi-direct product:
\begin{equation} \label{prod}
\sigma(g_0 \ldots,  g_{d-1}) \cdot \tau(h_0 \ldots,  h_{d-1}) = \sigma \tau(g_{\tau(0)}h_0 \ldots,  g_{\tau(d-1)}h_{d-1}),
\end{equation}
where $g_i = g|_i$ and $h_i = h|_i$.  
If we take $v \in X^*$ of length $n$, we may consider \eqref{prod} as giving the wreath recursion of $g, h \in \Aut(X^*)$ acting on $X^n$.  This gives 
\begin{equation} \label{singleelt}
(gh)(v) = g(h(v)) \qquad \text{and} \qquad (gh)|_v = g|_{h(v)} \cdot h|_v
\end{equation}

\begin{example} \label{counter1}
Let $d = 2$ and take $\sigma$ to be the non-trivial element of $S_2$.  Let $a = (a, b)$, $b = \sigma(1,1)$ and $G = \langle a, b \rangle$.  From \eqref{prod}, we have $b^2 = \sigma^2(1,1) = 1$ and $a^2 = (a^2, b^2) = (a^2, 1)$.  By induction this gives $a^2 = 1$.  However, the element $ba = \sigma(b,a)$ is spherically transitive, i.e., acts on each $X^n$ as a $2^n$-cycle, and in particular has infinite order.   This is a consequence of Proposition \ref{sphertrans}.  In Section \ref{background2} we show that $G$ is isomorphic to the iterated monodromy group of the Chebyshev polynomial $z^2 - 2$.  
\end{example}

\begin{example} \label{basilica1}
Let $d = 2$ and take $\sigma$ to be the non-trivial element of $S_2$.  Let $a = \sigma(1, b)$, $b = (1,a)$ and $G = \langle a, b \rangle$.  
This is the Basilica group, mentioned on page \pageref{basilica}.   If we write $X^2 = \{00, 01, 10, 11\}$, then from \eqref{lifts}, the wreath recursion for $a$ acting on $X^2$ is $\tau(1, 1, 1, a),$
where $\tau = (00,10)(01,11).$  Hence from \eqref{prod}, $a^2$ acts on $X^2$ as $(1, a, 1, a)$.  It follows that the restrictions of $a^n$ to words of length 2 are all of the form $a^k$ for $k < n$.  If $a$ is torsion of order $n$, then all restrictions of $a^n$ are trivial, and so $a^k = 1$ for some $k < n$, a contradiction.  Hence $a$ has infinite order, though it is not spherically transitive.
\end{example}

As an illustration of the preceding ideas, we give a characterization of spherically transitive elements of $\Aut(X^*)$. The proof is left as an exercise. 
\begin{proposition} \label{sphertrans}
Let $X$ have $d$ elements and $g \in \Aut(X^*)$.  
For each $v \in X^{n-1}$, let $\gamma_v \in S_d$ denote the action of $g|_v$ on $X$, 
and let $\rho_n = \prod_{v \in X^{n-1}} \gamma_v$.  Then $g$ is spherically transitive if and only if $\rho_n$ is a $d$-cycle for every $n \geq 1$.
\end{proposition}

\begin{remark}
Note that by convention $X^0 = \emptyset$, and $\gamma_\emptyset = \rho_1$ is the action of $g$ on $X$.   In the case $d = 2$, $\rho_n$ is the identity precisely when the number of $v \in X^{n-1}$ with $\gamma_v \neq 1$ is even.  Thus the Lemma says that $g$ is spherically transitive when $\gamma_v \neq 1$ for an odd number of $v \in X^{n-1}$, for all $n \geq 1$.  For the element $ba = \sigma(b,a)$ in Example \ref{counter1}, it is easy to see that $\gamma_v \neq 1$ for only one $v$ in each $X^{n-1}$.  
\end{remark}

%
%

\subsection{Self-similar and contracting groups}
A group $G \leq \Aut(X^*)$ is {\em self-similar} if $g|_x \in G$ for all $g \in G$ and $x \in X$.  We call $G$ {\em contracting} if there is a finite set $\mathcal{N} \subset G$ such that for every $g \in G$, $g|_v \in \mathcal{N}$ for all $v \in X^*$ sufficiently long.  The smallest set satisfying this condition is called the {\em nucleus} of the group.  In contracting groups, one can reduce many computations in $G$ to considerations involving only a finite set.  For instance, as pointed out in \cite{bartnek}, solving the so-called word problem (determining whether a given product of $n$ generators is trivial) can be done in polynomial time in a contracting group.  

We now consider the set of \textit{stable} elements of $G$,
$$\mathcal{N}_0 = \{g \in G : \text{$g|_v = g$ for some non-empty $v \in X^*$}\}.$$  
\begin{proposition} \label{contracting}
If $G \leq \Aut(X^*)$ is contracting, then $\mathcal{N}_0$ is finite and the nucleus of $G$ is equal to
\begin{equation} \label{theset}
\{h \in G : \text{$h = g|_w$ for some $g \in \mathcal{N}_0, w \in X^*$}\}.
\end{equation}
\end{proposition}

\begin{proof}
By definition, the nucleus of $G$ consists of the elements of $g$ for which there exists $r \in G$ with $r|_w = g$ for arbitrarily long words $w$.  If $g|_v = g$ for some non-empty $v$ and $v^n$ is the $n$-fold concatenation of $v$ with itself, then from \eqref{lifts} we have $g|_{v^n} = g$ for all $n \geq 1$.  Moreover, any $h$ with $g|_w = h$ for some $w \in X^*$ must also occur as the restriction of $g$ at arbitrarily long words.  Hence the set in \eqref{theset} is contained in the nucleus, and in particular $\mathcal{N}_0$ is finite.  On the other hand, if $h$ is in the nucleus, let $r \in G$ with $r|_w = h$ for arbitrarily long words $w$.  Let $n_1$ be the size of the nucleus and $n_2$ be such that $r|_u$ is in the nucleus when $u$ has length at least $n_2$.  We may take the length of $w$ to exceed $n_1 + n_2$.  Hence if $w_k$ is the length-$k$ initial word of $w$, then $r|_{w_k}$ is in the nucleus for more than $n_1$ values of $k$, and hence $r|_{w_k} = r|_{w_j}$ for some $k < j$.  Therefore $r|_{w_k} \in \mathcal{N}_0$ and there is a word $w'$ with $r|_{w_k}|_{w'} = r|_w = h$.   
\end{proof}

It is known that standard actions on $X^*$ of iterated monodromy groups of post-critically finite polynomials are always contracting \cite[Theorem 6.4.4]{nek}, and Proposition \ref{kneadstab} gives a method for computing $\mathcal{N}_0$ for a class of groups including iterated monodromy groups.  
For the group from Example \ref{counter1}, we have that $\mathcal{N}_0 = \{1, a\}$, and hence $G$ has nucleus $\{1, a, b\}$.  For the Basilica group (Example \ref{basilica1}), we have $\mathcal{N}_0 = \{1, a, b, a^{-1}, b^{-1}, ba^{-1}, ab^{-1}\}$ (see the remark following Proposition \ref{kneadstab}), and in this case $\mathcal{N}_0$ coincides with the nucleus.

\section{The fixed-point process} \label{fpprocess}

As noted in the introduction, the profinite completion $G_\infty$ of $G$ with respect to the $G_n$ comes equipped with a natural probability measure that projects to the discrete measure on each $G_n$.   In this section we define a stochastic process -- that is, an infinite collection of random variables defined on a common probability space -- that encodes information about the number of fixed points in $X^n$ of elements of $G_n$.  We then adapt techniques of \cite{galmart} to show that this process is a martingale provided that $G$ contains a spherically transitive element.  Finally, we apply a martingale convergence theorem that leads to the proofs of Theorem \ref{gen}, Theorem \ref{crystal}, and Corollary \ref{torfree}.    

Given $g \in G$ where the group $G$ acts naturally on a set $S$, we denote by $\Fix(g)$ the number of elements of $s$ with $g(s) = s$.   Define a stochastic process $Y_1, Y_2, \ldots$ on $G_\infty$ by taking $Y_i(g) = \#\Fix(\pi_i(g))$, where $\pi_i$ is the natural projection $G_\infty \to G_n$ and $G_n$ acts on $X^n$.  We call this the {\em fixed point process} of $G$, and write it $FP(G)$.  Because $\mu(\pi_i^{-1}(T)) = \#T/\#G_i$ for any $T \subseteq G_i$, we have that $\mu(Y_1 = t_1, \ldots, Y_n = t_n)$ is given by 
\begin{equation} \label{fpchar}
\frac{1}{\# G_n}\# \left\{g \in G_n : \mbox{$g$ fixes $t_i$ elements of $X^i$ for 
$i = 1,2, \ldots, n$} \right\}.
\end{equation}
We denote by $E(Y)$ the expected value of the random variable $Y$.

\begin{definition}
A stochastic process with probability measure $\mu$ and random variables $Y_1, Y_2, \ldots$ taking values in $\mathbb{R}$ is a {\em martingale} if for all $n \geq 2$ and any $t_i \in \mathbb{R}$, 
$$E(Y_n \mid Y_{1} = t_{1}, Y_2 = t_2, \ldots, Y_{n-1} = t_{n-1}) = t_{n-1},$$
provided $\mu(Y_{1} = t_{1}, Y_2 = t_2, \ldots, Y_{n-1} = t_{n-1}) > 0$.  
\end{definition}

\begin{theorem} \label{mart}
Let $G \leq \Aut(X^*)$ have a spherically transitive element.  Then $FP(G)$ is a martingale.  
\end{theorem}

\begin{proof}
We must show that 
\begin{equation} \label{sloop}
E(Y_{n} \mid Y_1 = t_1, \ldots, Y_{n-1} = t_{n-1}) = t_{n-1},
\end{equation}
where $t_1, \ldots, t_{n-1}$ satisfy $\mu (Y_1 = t_1, \ldots, Y_{n-1} = t_{n-1}) > 0.$  Because the $Y_i$ take integer values, each $t_i$ must be an integer.  
By definition, the left-hand side of \eqref{sloop} is
\begin{equation} \label{sloop1}
\sum_k k \cdot \frac{\mu (Y_1 = t_1, \ldots, Y_{n-1} = t_{n-1}, Y_n = k )}
{\mu (Y_1 = t_1, \ldots, Y_{n-1} = t_{n-1})}.
\end{equation}
Put 
\begin{eqnarray*}
S & = & \{g \in G_n : \text{$g$ fixes $t_i$ elements of $X^i$ for $1 \leq i \leq n-1$}\} \\
S_k & = & \{g \in S : \text{$g$ fixes $k$ elements of $X^n$}\} 
\end{eqnarray*}
By \eqref{fpchar}, the expression in \eqref{sloop1} is equal to
$\sum_k k \cdot (\# S_k/\# S)$.
This in turn may be rewritten
\begin{equation} \label{sloop4}
\frac{1}{\#S} \sum_{g \in S} \#\Fix(g).
\end{equation}

Let $\sigma \in G_n$ be the image under $\pi_n$ 
of the spherically transitive element of $G_\infty$ assumed to exist.  Then $\tau = \sigma^{d^{n-1}}$ acts trivially on $X^{n-1}$, and hence $S$ is invariant under multiplication by powers of $\tau$, and therefore is a disjoint union of cosets of $\langle \tau \rangle$.  Note that because $\langle \sigma \rangle$ acts transitively on $X^n$, $\langle \tau \rangle$ must act transitively on each set $v* = \{vx : x \in X\}$ for $v \in X^{n-1}$.    

Now take $g\langle \tau \rangle \subseteq S$, and let $R = \{vx : v \in X^{n-1}, g(v) = v, x \in X\}$ be the set of elements of $X^n$ lying above elements of $X^{n-1}$ fixed by $g$.  Note that because $g \in S$, we have $\#R = dt_{n-1}$.  If $vx \in R$, then $g(vy) = vx$ for some unique $y \in X$.  There is a unique $i \in \{0, \ldots, d-1\}$ such that $\tau^i(vx) = vy$, and thus $g\tau^i(vx) = vx$.  If $I(g,s)$ is the function that takes the value $1$ when $g(s) = s$ and $0$ otherwise, we have shown that $\sum_{i=0}^{d-1} I(g\tau^i, vx) = 1$ and hence
\begin{equation*}
\sum_{vx \in R} \sum_{i=0}^{d-1} I(g\tau^i, vx) = dt_{n-1}.
\end{equation*}
Inverting the order of summation and using that $g(w) \neq w$ for $w \not\in R$, we have
$$
\sum_{i = 0}^{d-1}\# \Fix(g\tau^i) = dt_{n-1}.
$$
But $S$ is the disjoint union of cosets of $\langle \tau \rangle$, and hence
$$
\sum_{g \in S} \# \Fix(g) = \#S \cdot t_{n-1}.
$$
Therefore the expression in \eqref{sloop4} equals $t_{n-1}$.
\end{proof}

Martingales are useful tools because they often converge in the following sense:
\begin{definition}
Let $Y_1, Y_2, \ldots$ be a stochastic process defined on the probability space $\Omega$ with probability measure $\mu$.  
The process {\em converges} if 
$$\mu \left(\omega \in \Omega : \text{$\inflim{n} Y_n(\omega)$ exists} \right) = 1.$$
\end{definition}
We give one standard martingale convergence theorem (see e.g. \cite[Section 12.3]{grimmett} for a proof).
\begin{theorem} \label{martconv}
Let $M = (Y_1, Y_2, \ldots)$ be a martingale whose random variables take nonnegative real values.  Then $M$ converges.
\end{theorem}
Since the random variables in $FP(G)$ take nonnegative integer values, we immediately have the following:  
\begin{corollary} \label{evconstcor}
Let $G \leq \Aut(X^*)$ contain a spherically transitive element.  Then 
$$\mu(\{g \in G_\infty : \text{$Y_1(g), Y_2(g), \ldots$ is eventually constant}\}) = 1.$$
\end{corollary}
In particular, any $g \in G_\infty$ fixing infinitely many ends of $X^*$ must have $Y_i(g) \to \infty$, and hence lie in a set of measure zero.  This proves Theorem \ref{gen}.  

We may now give a short proof of Theorem \ref{crystal}.  Assume the hypotheses of that theorem, and let $\mathcal{N}$ be the nucleus of $G$.  Suppose that $g \in G_{\infty}$ fixes some end $w = x_1x_2\cdots$ of $X^*$.  Let $v_n = x_1x_2 \cdots x_n$ for each $n \geq 1$, and consider the sequence of restrictions $g|_{v_1}, g|_{v_2}, \ldots$.  For $n$ large enough, we have $g|_{v_n} \in \mathcal{N}$, and $g|_{v_n}$ fixes the end $x_{n+1}x_{n+2} \cdots$ since $g$ fixes $w$.  Because $\mathcal{N}$ is finite, there must be $i < j$ with $g|_{v_i} = g|_{v_j}$.  Let $h = g|_{v_i}$, and note that for $w = x_{i+1}x_{i+2} \cdots x_j$ we have $h(w) = w$ and $h|_w = h$.  Hence $h \in \mathcal{N}_1$, and by hypothesis fixes infinitely many ends of $X^*$.   Inserting $v_i$ on the beginning of each of these ends, we obtain infinitely many ends of $X^*$ fixed by $g$.  Hence by Corollary \ref{evconstcor}, $g$ lies in a set of measure zero, proving the theorem.  

To derive Corollary \ref{torfree}, note that if $g \in \mathcal{N}_1$, then $g(v) = v$ and $g|_v = g$ for some non-empty $v \in X^*$.  From \eqref{singleelt} it follows that $g^n(v) = v$ and $g^n|_v = g^n$ for all $n \geq 1$, and hence $g^n \in \mathcal{N}_1 \subseteq \mathcal{N}_0$ for all $n \geq 1$.  Because $G$ is contracting, $\mathcal{N}_0$ is finite by Proposition \ref{contracting}, and thus two distinct powers of $g$ are equal, implying that $g$ is torsion.  Therefore if $G$ is torsion-free then $\mathcal{N}_1$ is trivial, and Corollary \ref{torfree} follows from Theorem \ref{crystal}.

\section{Background and Examples, part II} \label{background2}

\subsection{Computations of iterated monodromy groups}  Recall from Section \ref{mono1} that if $f$ is a post-critically finite polynomial with post-critical set $P_f$, then $\img(f)$ acts naturally on the tree $T_\beta = \bigsqcup_{n \geq 1} f^{-n}(\beta)$ of preimages of any $\beta \in \C \setminus P_f$.  If $f$ has degree $d$, then we may take $X = \{0, 1, \ldots, d-1\}$, and choose a bijection $\Lambda : X \to f^{-1}(\beta)$.  This extends to an isomorphism $\Lambda : X^* \to T$ (\cite[Proposition 5.2.1]{nek}) that conjugates the action of $\img(f)$ to that of some $G \leq \Aut(X^*)$ on $X^*$.  We call this a {\em standard action} of $\img(f)$ on $X^*$, and it gives an explicit way to compute a recursive formula for elements of $\img(f)$ in the form of a wreath recursion \cite[Proposition 5.2.2]{nek} (see also \cite[Proposition 2.2]{bartnek}).   

The action of $\img(f)$ on $T$ is generated by the action of the generators of $\pi_1(\C \setminus P_f)$ on $T$.  For each $z_0 \in P_f$ there is a generator of $\pi_1(\C \setminus P_f)$, and under a standard action there is a corresponding $g_{z_0} \in \Aut(X^*)$.  

The next result follows from \cite[Theorem 6.8.3]{nek}.   For $f \in \C[z]$ and $y \in \C$, denote by $\ord_f(y)$ the order of vanishing of $f(z) - f(y)$.  Clearly $\ord_f(y) \geq 1$, with $\ord_f(y) > 1$ if and only if $y$ is a critical point of $f$.    

\begin{theorem} \label{compmon}
Let $f \in \C[z]$ be a post-critically finite polynomial, with post-critical set $P_f$.  Let $G \leq \Aut(X^*)$ be a standard action of $\img(f)$ on $X^*$, and for $z_0 \in P_f$ let $g = g_{z_0} \in G$ be the element corresponding to $z_0$.   

Then the action of $g$ on $X$ contains one $m$-cycle for each $c \in f^{-1}(z_0)$ with $\ord_f(c) = m$.  Let $x_1, x_2, \ldots, x_m$ be the cycle corresponding to $c$. 
If $c \not\in P_f$, then $g|_{x_i} = 1$ for each $i = 1, \ldots, m$.  If $c \in P_f$, then there is a unique $i \in \{1, \ldots, m\}$ such that $g|_{x_i}$ is the element of $G$ corresponding to $c$, and $g|_{x_i} = 1$ otherwise.   
\end{theorem}

\begin{remark}
Although we don't regard $\infty$ as being in $P_f$, Theorem \ref{compmon} nonetheless applies to it.  Because $f^{-1}(\infty) = \{\infty\}$, it is a point of multiplicity $d$, and we have that $g_\infty$ acts as a $d$-cycle on $X$, with restriction to some $x \in X$ giving $g_{\infty}$ and the other restrictions being trivial.  It follows from Proposition \ref{sphertrans} that $g_\infty$ is spherically transitive.  The fact that $G$ contains a spherically transitive element is also a consequence of Theorem \ref{char} and Lemma \ref{trans}. 
\end{remark}

As an illustration of this result, we show that the group in Example \ref{counter1} is a standard action of the iterated monodromy group of $f(z) = z^2 - 2$ on $X^*$, where $X = \{0,1\}$.  We have $P_f = \{-2, 2\}$, so that $G = \langle g_{-2}, g_2 \rangle$.  Now $f^{-1}(-2) = \{0\}$ and $\ord_f(0) = 2$, implying that $g_{-2}$ acts on $X$ as a 2-cycle.  Because $0 \not\in P_f$, the restrictions of $g_{-2}$ are trivial.  On the other hand $f^{-1}(2) = \{-2, 2\}$, so $g_2$ acts trivially on $X$.  Because $-2 \in P_f$ but $2 \not\in P_f$, the restriction of $g_2$ to one element of $X$ is trivial, while the other one is $g_{-2}$.  Either choice gives the same group up to conjugacy in $X^*$ (indeed, up to conjugacy in $G$, since conjugating by $g_{-2}$ exchanges the restrictions of $g_2$).  

\subsection{Automata and Moore diagrams}
A very useful description of $g \in \Aut(X^*)$ in terms of its wreath recursion comes via automata theory.  The set $Q(g) = \{g|_v : v \in X^*\}$ of all restrictions of $g$ may be viewed as the set of states of an automaton.  Being in a state $g|_w$ for some $w \in X^n$ and receiving an input letter $x \in X$, the automaton types on the output tape $g|_w(x)$ and proceeds to the state $(g|_w)|_x$, which by \eqref{lifts} is just $g|_{wx}$.  In this way the action of $g$ on any $v \in X^*$ may be determined.  
We formalize this in the following definition:
\begin{definition}
An {\em automaton} $A$ over the set $X$ is given by 
\begin{itemize}
\item the set of states, which we denote also by $A$;
\item a map $\tau: A \times X \to X \times A$.
\end{itemize}
If $\tau(a,x) = (y,b)$, then $y$ and $b$ as functions of $(a,x)$ are called the {\em output} and {\em transition function}, respectively.  We say that $A$ is {\em invertible} if each $a \in A$ acts on $X$ as a permutation. 
\end{definition}

The {\em Moore diagram} of an automaton $A$ provides a good method of visualization.  It is a directed labeled graph whose vertex set is the set $A$ of states of the automaton.   If $\tau(a,x) = (y,b)$, then there is an arrow from $a$ to $b$ labeled by $(x,y)$.  If $A$ is invertible, the Moore diagram of the inverse automaton is given by formally replacing each state $a$ by $a^{-1}$ and changing each arrow labeling from $(x,y)$ to $(y,x)$.  Given an automaton $A$ over a set $X$, it is easy to see that the states of $A$ define elements of $\Aut(X^*)$.  Indeed, we can recover the wreath recursion for $a \in A$ by 
noting that if $\tau(a,x) = (y,b)$ then $a(x) = y$ and $a|_x = b$.  In this case we say that $G = \langle A \rangle$ is generated by the automaton $A$.  

By Theorem \ref{compmon}, a standard action of the iterated monodromy group of a post-critically finite polynomial is generated by a set that is closed under restrictions.  Hence the automaton generating such a group is finite.  See Figure 
\ref{fig: bas} for an example.  


\begin{figure}[htbp]
\begin{center}
\includegraphics[height = 2.5in]{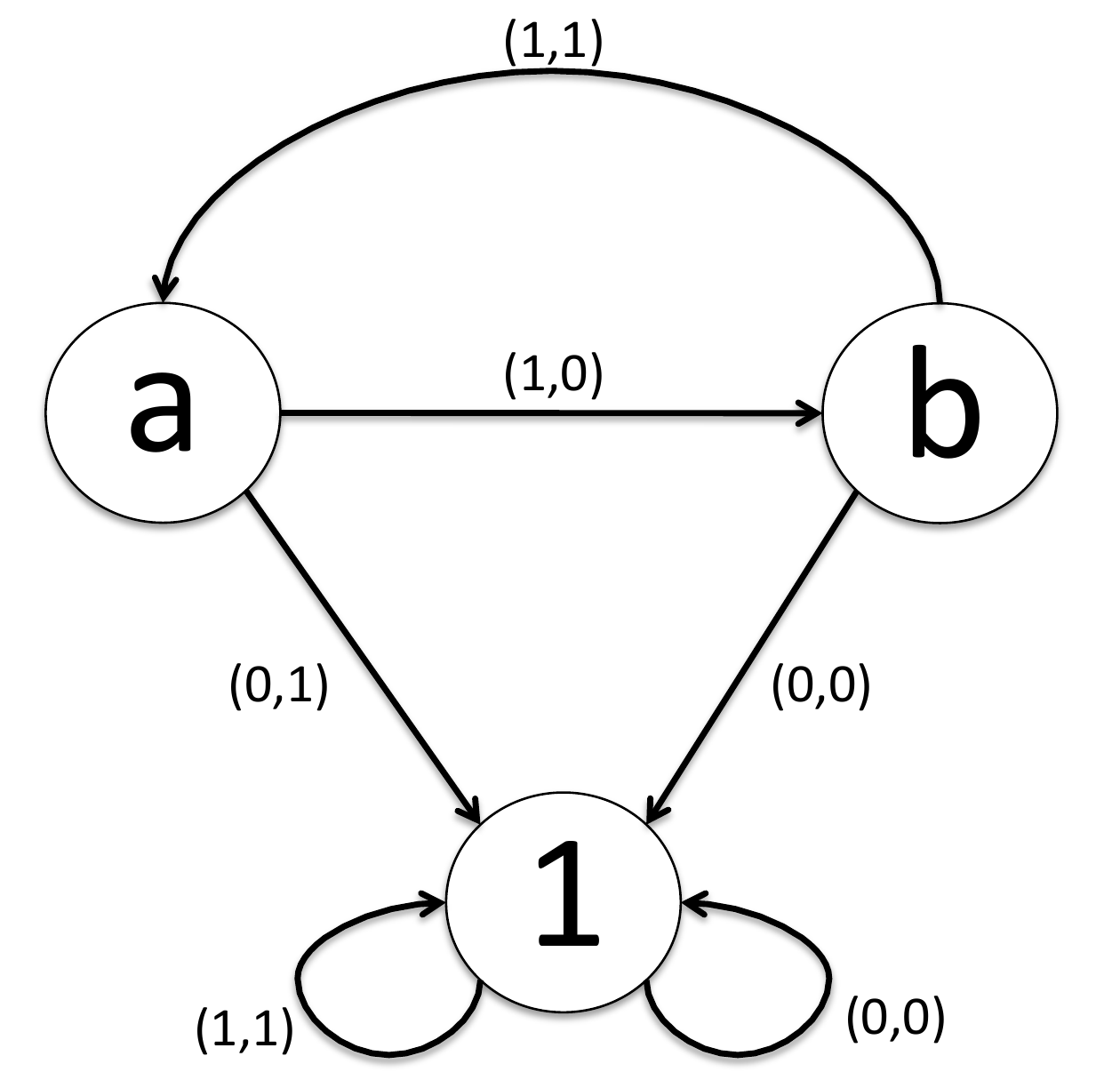}
\end{center}
\caption{Moore diagram of the automaton generating the Basilica group; see Example \ref{basilica1}. The state labeled 1 is the identity state, and is distinct from the element $1 \in X$.}
\label{fig: bas}
\end{figure}


\subsection{Bounded and finitary automorphisms}
\begin{definition} \label{bounded}
We say that $g \in \Aut(X^*)$ is {\em finite-state} if it is defined by a finite automaton, or equivalently if $\{g|_v : v \in X^*\}$ is a finite set.  We call $g$ {\em bounded} if it is finite-state and the sequence $$q_n = \#\{v \in X^n : g|_v \neq 1\}$$ is bounded.  We call $g$ {\em finitary} if $q_n = 0$ for all $n$ sufficiently large, or equivalently if there exists $n_0$ such that $g|_v$ is trivial for all words of length at least $n_0$.  
\end{definition}

Finitary automorphisms will play a major role in Sections \ref{sec: kneadgraph} and \ref{sec: last}.  The main fact we will use about the more general notion of bounded automorphisms is the following special case of a theorem of Nekrashevych and Bondarenko:
\begin{theorem}{\cite{bondnek}, \cite[Theorem 3.9.12]{nek}} \label{bndedthm}
Let $G \leq \Aut(X^*)$ be generated by a finite automaton whose states define bounded automorphisms of $X^*$.  Then $G$ is contracting.  
\end{theorem}

\subsection{Kneading automata and theorem of Nekrashevych}
We require a strong result of Nekrashevych that characterizes the $G \leq \Aut(X^*)$ that are isomorphic to a standard action of the iterated monodromy group of a post-critically finite polynomial.  This characterization is purely in terms of a finite automaton that generates $G$.   
To state this result, we require the notion of a {\em tree-like multi-set of permutations}.  Recall that a multi-set of permutations of a set $X$ is a map $i \mapsto \pi_i$ from a set $I$ of indices to the set $\Sym(X)$ of permutations of $X$.  Thus for instance distinct indices may give the same permutation.  We denote the set $\{\pi_i : i \in I\}$ by $T$.  The {\em cycle diagram} associated to $T$ is an oriented 2-dimensional CW-complex whose set of $0$-cells is $X$.  For each cycle $(x_1, x_2, \ldots, x_n)$ of each $\pi_i \in T$, there is a 2-cell whose boundary passes through $x_1, x_2, \ldots x_n$ and no other elements of $X$, and whose order on the boundary corresponds to the order in the cycle.  Two different  2-cells can only intersect at 0-cells.  We call the {\em reduced cycle diagram} of $T$ the diagram obtained by deleting the 2-cells corresponding to fixed points of the $\pi_i$.  

\begin{definition}
A multi-set $T$ of permutations of a set $X$ is said to be {\em tree-like} if the cycle diagram of $T$ is contractible.
\end{definition}

For an example of a tree-like multi-set, see Figure \ref{fig: cycdiag}.  Note that we could add the identity to this multi-set any number of times and it would still be tree-like.  However, adding any non-trivial element of $S_6$ would yield a non-tree-like multi-set.  
\begin{figure}[htbp]
\begin{center}
\includegraphics[width = 5in]{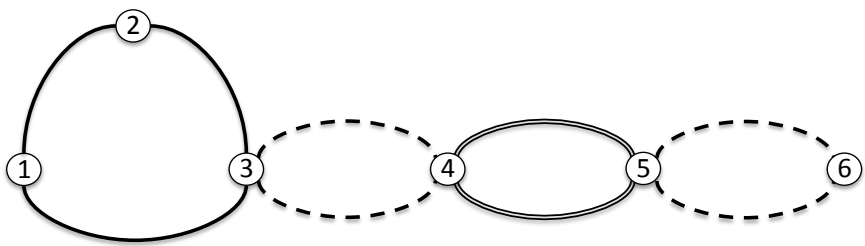}
\end{center}
\caption{The reduced cycle diagram of the multi-set $\{(1,2,3), (3,4)(5,6), (4,5)\}$ of elements of $S_6$.  The arrows are omitted, and the action of the elements are given by the solid, dashed, and doubled lines, respectively.  This multi-set is tree-like.}
\label{fig: cycdiag}
\end{figure}
Another way to visualize the action of a multi-set of permutations $T$ on a set $X$ is via its {\em cycle graph}.  We define it to be a bipartite graph obtained from the reduced cycle diagram by coloring each vertex of the former white, and replacing each 2-cell by a black vertex connected to the white vertices on the boundary of the 2-cell.  See Figure \ref{fig: cycgraph} for the cycle graph corresponding to the multi-set from Figure \ref{fig: cycdiag}.   Note that our definition differs slightly from that of \cite[p.  186]{nek}, where the cycle graph is not defined to be bipartite, but is otherwise identical.  
\begin{figure}[htbp]
\begin{center}
\includegraphics[width = 5in]{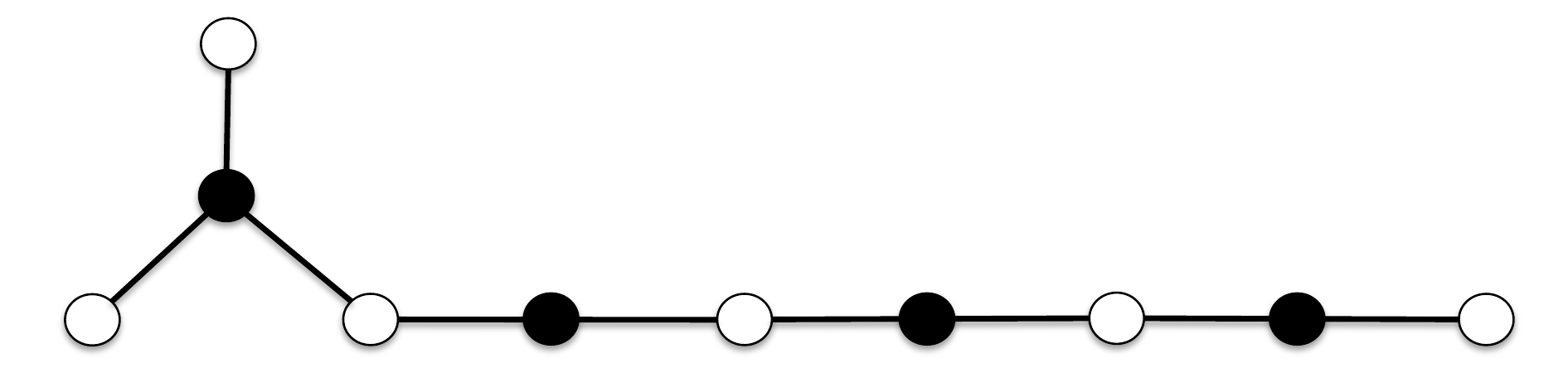}
\end{center}
\caption{The cycle graph of the multi-set of permutations given in Figure \ref{fig: cycdiag}}
\label{fig: cycgraph}
\end{figure}

The cycle graph and cycle diagram are clearly homotopically equivalent, and thus a multi-set of permutations is tree-like if and only if its cycle graph is a tree.  In Section \ref{sec: treelike} we give several results on tree-like sets of permutations.  

We may now state the characterization of iterated monodromy groups:

\begin{theorem}{\cite[Theorem 6.10.8]{nek}}  \label{char}
A subgroup $G \leq \Aut(X^*)$ is isomorphic to a standard action of the iterated monodromy group of a post-critically finite polynomial if and only if $G$ is the group generated by a finite invertible automaton $A$ with the following properties:
\begin{enumerate}
\item For each non-trivial $a \in A$, there is a unique arrow into the state $a$.  In other words, there is a unique $b \in A$ and $x \in X$ with $b|_x = a$.
\item For each $a \in A$ and each cycle $(x_1, x_2, \ldots, x_n)$ of the action of $a$ on $X$, the restriction $a|_{x_i}$ is non-trivial for at most one $x_i$.
\item The multi-set of permutations defined by the set of states of $A$ acting on $X$ is tree-like.  
\item Let $a_1 \neq a_2$ be non-trivial states of $A$ with $v_1, v_2 \in X^* \setminus 
\{ \emptyset \}$ satisfying $a_i|_{v_i} = a_i$ and $a_i(v_i) = v_i$ for $i = 1,2$.  Then there is no $h \in G$ with $h(v_1) = v_2$ and $h|_{v_1} = h$.   
\end{enumerate}
\end{theorem}
For example, the automaton given in Figure \ref{fig: bas} satisfies all the conditions of Theorem \ref{char}.  
We do not use even close to the full strength of Theorem \ref{char}.  Indeed, we require only the far easier direction, which is that if $G$ is isomorphic to a standard action of an iterated monodromy group, then $G = \langle A \rangle$, where $A$ satisfies conditions (1)-(4).  Moreover, we do not use condition (2).  

We introduce a definition following the terminology of \cite{nek}:
\begin{definition} \label{kneadingdef}
A {\em kneading automaton} is a finite invertible automaton satisfying conditions (1)-(3) of Theorem \ref{char}.
\end{definition}

\section{Results on tree-like sets of permutations} \label{sec: treelike}

In this section we present several results that will play roles in the proofs of our mains theorems.  The first two appear in \cite{nek}.


\begin{lemma}{\cite[Proposition 6.7.5]{nek}}  \label{neklem2}
Let $A$ be a kneading automaton.  Then for any $n \geq 1$, the multi-set of permutations defined by the states of $A$ acting on $X^n$ is tree-like.    
\end{lemma}

\begin{lemma}{\cite[Corollary 6.7.7]{nek}} \label{trans}
If $A$ is a kneading automaton, then the product of the states of $A$ (taken in any order) is a spherically transitive element of $\Aut(X^*)$.   
\end{lemma}

\begin{lemma} \label{treefixed}
Let $T = \{\pi_1, \ldots, \pi_n\}$ be a tree-like multi-set of permutations of a set $X$, and let $\sigma = \prod_{i \in I} \pi_i$ for some non-empty $I \subseteq \{1, \ldots, n\}$.  Suppose that $\sigma(x) = x$ for some $x \in X$.  Then $\pi_i(x) = x$ for all $i \in I$. 
\end{lemma}

\begin{proof}
Induct on $\#I$.  When $\#I = 1$ the statement is trivial.  Suppose that $\#I \geq 2$ and $\left(\prod_{i \in I} \pi_i \right)(x) = x$, and let $k \in I$ and $$\sigma_k = \left(\prod_{i \in I \setminus \{k\}} \pi_i \right).$$
If $\sigma_k(x) = y \neq x$, then necessarily $\pi_k(y) = x$.  Thus in the cycle graph of $T$ there is a path from the white vertex corresponding to $x$ to the white vertex corresponding to $y$, given by the action of $\sigma_k$.  There is a distinct path from the vertex corresponding to $y$ back to the vertex corresponding to $x$, given by the action of $\pi_k$.  This contradicts the hypothesis that the cycle graph is a tree.  


Therefore $\sigma_k(x) = x$, and hence $\pi_k(x) = x$.  Applying the inductive hypothesis to $\sigma_k$ gives that $\pi_i(x) = x$ for all $i \in I$. 
\end{proof}

\begin{lemma} \label{treecyc}
Let $T$ be a tree-like multi-set of permutations acting on a set $X$ with $\#X = d \geq 2$.  Then the reduced cycle diagram of $T$ has at most $(d-1)$ 2-cells, with equality if and only if every element of $T$ acts on $X$ as a (possibly empty) disjoint product of transpositions.  
\end{lemma}

\begin{proof}
We induct on $d$.  If $d = 2$, then the reduced cycle diagram of $T$ has a single 2-cell, and the unique element of $T$ acting non-trivially on $X$ acts as a transposition.  Hence the lemma holds.  Assume that $d \geq 3$, and consider the cycle graph of $T$.  Because it is a tree, there must exist a vertex $v$ of degree 1 (a {\em leaf} of the tree).  This vertex must be white, since the black vertices by definition correspond to cycles and so have degree greater than one.  Note that $v$ is connected to a unique black vertex $b$, and hence is fixed by all but one element $\pi_i$ of $T$.  Consider the element $\pi'_i$ obtained by deleting from $\pi_i$ the cycle containing $v$.  Replacing $\pi_i$ by $\pi'_i$ gives a new multi-set $T'$ whose cycle graph is the same as that of $T$, except that $b$ and all leaves connected to $b$ have been deleted.  Note this results in deleting at least one white vertex, namely $v$, and this is the only white vertex deleted if and only if the deleted cycle of $\pi_i$ was a 2-cycle.  

Thus $T'$ is tree-like and acts on a set $X'$ with $\#X' \leq \#X - 1$; moreover we have equality if and only if the only cycle in an element of $T$ that is not in an element of $T'$ is a 2-cycle.  We may apply the inductive hypothesis to get that there are at most $\#X' - 1$ black vertices in the cycle graph of $T'$, with equality if and only if all elements of $T'$ are (possibly empty) disjoint products of transpositions.  But this cycle graph contains exactly one fewer black vertex than the cycle graph of $T$, and hence the latter has at most $\#X'$ black vertices, with equality if and only if all elements of $T$ are (possibly empty) products of disjoint 2-cycles.  The number of black vertices in the cycle graph of $T$ is by definition the same as the number of 2-cells in the reduced cycle diagram of $T$. 
\end{proof}

\begin{lemma} \label{treeperms}
Let $T = \{\pi_1, \ldots, \pi_n\}$ be a tree-like multi-set of permutations of a set $X$ with $\#X = d$.  
\begin{enumerate}
\item For any $i \neq j$, we have $\#\Fix(\pi_i) + \#\Fix(\pi_j) \geq 2$.
\item If $\#\Fix(\pi_i) + \#\Fix(\pi_j) \leq 3$ for some $i \neq j$, then $\pi_k$ is the identity for all $k \not\in \{i, j\}$.  
\end{enumerate}
\end{lemma}

\begin{proof}
We begin by noting that by definition the cycle graph (and thus the reduced cycle diagram) of $T$ is a contractible tree, and hence connected.  If the cycle graph (equivalently, reduced cycle diagram) of some subset $S$ of $T$ is also connected, then the two cycle graphs must coincide, and hence all elements of the multi-set $S \setminus T$ are the identity.  

Consider the reduced cycle diagram of the multi-set $\{ \pi_i, \pi_j\}$, where $i \neq j$.  It is a (possibly disconnected) planar graph, and hence by Euler's formula satisfies 
\begin{equation} \label{euler}
V - E + F = 2 + (c-1), 
\end{equation} 
where $V, E,$ and $F$ denote the numbers of vertices, edges, and faces (counting the face at infinity), respectively, and $c$ denotes the number of connected components of the graph.  Now the vertex set is just $X$, so $V = d$.  There are $r$ edges for each $r$-cycle of $\pi_i$ or $\pi_j$, where $r > 1$ (recall that fixed points do not appear in the reduced cycle diagram).  Thus $E = (d - \#\Fix(\pi_i)) + (d- \#\Fix(\pi_j))$.  Finally, there is one face for each cycle of $\pi_i$ or $\pi_j$, plus the face at infinity.  Because the reduced cycle diagram of $\{ \pi_i, \pi_j\}$ is a subset of the reduced cycle diagram for $T$, we have from Lemma \ref{treecyc} that $\pi_i$ and  $\pi_j$ have at most $d-1$ cycles between them, and hence $F \leq d-1 + 1 = d$.   Therefore \eqref{euler} gives
\begin{equation} \label{euler2}
2d - (\#\Fix(\pi_i) + \#\Fix(\pi_j)) + 2 = E + 2 = V + F - (c-1) \leq 2d - (c-1),
\end{equation}
and assertion (1) follows.  

Note that in \eqref{euler2} we have equality if and only if $F = d$, which occurs precisely when the cycle diagram of $\{ \pi_i, \pi_j\}$ has $d-1$ 2-cells.  When this happens, we have by Lemma \ref{treecyc} that the number of 2-cells of the cycle diagram of $\{ \pi_i, \pi_j\}$ is the same as the number of 2-cells of the cycle diagram of $T$, and hence the two diagrams coincide.  It follows that $c = 1$.  We have thus shown that either $c = 1$ or 
\begin{equation} \label{euler3}
2d - (\#\Fix(\pi_i) + \#\Fix(\pi_j)) + 2 \leq 2d - c.
\end{equation}
Now \eqref{euler3} implies that either $c = 1$ or $c \leq \#\Fix(\pi_i) + \#\Fix(\pi_j) - 2$.  In particular, either $c = 1$ or $\#\Fix(\pi_i) + \#\Fix(\pi_j) \geq 4$.  This together with the remarks at the beginning of the proof establish assertion (2).  
\end{proof}

\section{Kneading graphs and the structure of $\mathcal{N}_1$} \label{sec: kneadgraph}



In this section we exploit condition (1) of Theorem \ref{char} and the results of Section \ref{sec: treelike} to study the set 
$$\mathcal{N}_1 = \{g \in G : \text{$g|_v = g$ and $g(v) = v$ for some non-empty $v \in X^*$}\}$$
first defined on p. \pageref{ndefs}.
 
Condition (1) of Theorem \ref{char} implies that if we delete the trivial state from the Moore diagram of a kneading automaton $A$ (along with all the arrows originating at the trivial state) then then the resulting graph is a disjoint union of cycles with trees attached to them.  We call such a diagram the {\em reduced Moore diagram} of $A$.  See Figure \ref{fig: kneadingaut}.  
\begin{figure}[htbp]
\begin{center}
\includegraphics[width = 5in]{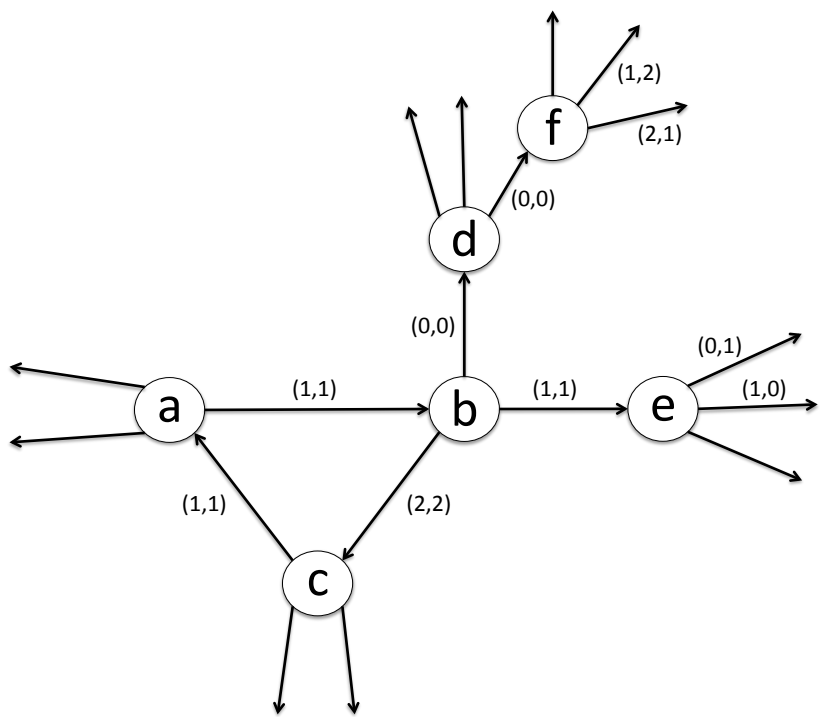}
\end{center}
\caption{The reduced Moore diagram of a kneading automaton over $X = \{0,1,2\}$ satisfying all the conditions of Theorem \ref{char}.  Labelings on arrows to the identity state have been deleted, except where the action on $X$ is non-trivial.}
\label{fig: kneadingaut}
\end{figure}

In particular, the states not in cycles have the property that all restrictions to sufficiently long words are the identity, and hence they define finitary automorphisms of $X^*$.  
To each state $a$ in a cycle of the Moore diagram we can associate its {\em kneading sequence} $x_1x_2\cdots \in X^{\omega}$, which is the unique infinite word such that for each $v_n := x_1\cdots x_n$, $a|_{v_n}$ belongs to the cycle containing $a$.  We refer to $v_n$ as the {\em length-$n$ kneading sequence} of $a$.  The (infinite) kneading sequence of any given state is periodic, with period dividing the length of the cycle in which the element lies.  For instance, for the automaton in Figure \ref{fig: kneadingaut}, the kneading sequences of $a$, $b$, and $c$ are $\overline{121}, \overline{211},$ and $\overline{112}$, respectively, where the bars denote repeating.  By hypothesis $A$ is invertible, and recall that the Moore diagram of the inverse automaton is given by replacing each state $a$ by $a^{-1}$ and changing each arrow labeling from $(x,y)$ to $(y,x)$.  Hence $a^{-1}$ is in a cycle of the Moore diagram of the inverse automaton of $A$ if and only if $a$ is in a cycle of the Moore diagram of $A$.  Each such $a^{-1}$ has a kneading sequence as before.  Let $C$ denote the collection of states of $A$ that are in cycles of the Moore diagram, together with their inverses. 

Let $m$ be the least common multiple of the periods of the kneading sequences of the elements of $C$.  The {\em kneading graph} of the automaton $A$ is the directed graph whose vertex set is the set of length-$m$ kneading sequences of the states belonging to $C$.     
There is a directed edge from $s_1$ to $s_2$ if $s_1$ is the length-$m$ kneading sequence for some $a \in C$ and $a(s_1) = s_2$.  We label such an edge with the element $a$.  Two kneading graphs are pictured in Figure \ref{fig: kneadinggraph}.
\begin{figure}[htbp]
\begin{center}
\includegraphics[width = 5in]{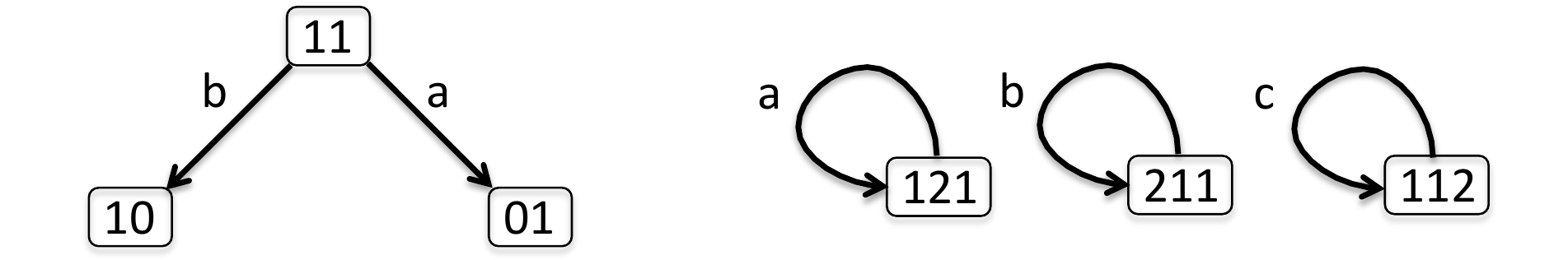}
\end{center}
\caption{The kneading graphs of the Basilica group (left) and the group generated by the automaton shown in Figure \ref{fig: kneadingaut}.}
\label{fig: kneadinggraph}
\end{figure}

Recall that the set of {\em stable} elements of $G \leq \Aut(X^*)$ is 
$$\mathcal{N}_0 = \{g \in G : \text{$g|_v = g$ for some non-empty $v \in X^*$}\}.$$
When $G$ is generated by an automaton $A$ satisfying the hypotheses of Theorem \ref{char}, the kneading graph of $A$ provides an algorithm for determining $\mathcal{N}_0$ and $\mathcal{N}_1$.  This idea first appeared in 
\cite[Lemma 3.2]{bartnek}, which deals with certain automata in the case $d = 2$.    

We require some terminology relating to the kneading graph.  By a {\em path} we mean any sequence $s_1, s_2, \ldots, s_n$ of vertices such that there is a directed edge from $s_{i-1}$ to $s_i$ or a directed edge from $s_i$ to $s_{i-1}$, for all $i = 2, \ldots, n$.  Note that this is more general than the usual notion of a path in a directed graph, since we permit paths to traverse edges against their direction.  We further stipulate that our paths have {\em no back-tracking}, that is, each edge traversed is either distinct from the previous edge, or is the same as the previous edge and also in the same direction (i.e. consists of going again around a cycle of length one).  By a {\em circuit}, we mean a path with a common starting and ending vertex; we allow repeats of vertices and edges.  A {\em cycle} is a circuit that repeats only its common starting and ending vertex.  
Recall that a kneading automaton is one satisfying conditions (1)-(3) of Theorem \ref{char}.
\begin{proposition} \label{kneadstab}
Let $G$ be generated by a kneading automaton $A$.  Then $\mathcal{N}_0$ consists of words in $A \cup A^{-1}$ obtained from the labels of paths in the kneading graph of $A$, where one reads the inverse of the labeled element if one follows an arrow backwards.  To assemble the word corresponding to a given path in the kneading graph, one copies the letters down from right to left.  

In addition, $\mathcal{N}_1$ consists of the words obtained from labels of ciruits in the kneading graph of $A$.
\end{proposition} 

\begin{remark}
For instance, the path of length 2 going from $01$ to $10$ in the kneading graph of the Basilica group (Figure \ref{fig: kneadinggraph}, left) gives $ba^{-1}$.  The other path of length 2 gives $ab^{-1}$, while the paths of length 1 yield $a, a^{-1}, b$, and $b^{-1}$.  Thus $\mathcal{N}_0$ consists of these six elements plus the identity.  Since there are no circuits in the kneading graph, $\mathcal{N}_1$ is trivial.  
\end{remark}

\begin{proof}
Recall that condition (1) of Theorem \ref{char} ensures that each $a \in A$ not in a cycle of the reduced Moore diagram is finitary, that is, has trivial restriction on all sufficiently long words in $X^*$.  If $a$ is in a cycle, the length-$n$ kneading sequence $v_n$ of $a$ is the unique word of length $n$ such that
$a|_{v_n}$ is not finitary.  We often simply call $v_n$ the kneading sequence of $a$ when the length $n$ is clear from context.  If $a \in A$ has kneading sequence $v_n$, then from \eqref{singleelt} we have
\begin{equation} \label{inv}
1 = (a^{-1}a)|_{v_n} = a^{-1}|_{a(v_n)}a|_{v_n}. 
\end{equation} 
If $h \in G$ is finitary and $g \in G$ is not, then for sufficiently large $k$ and any $v \in X^k$, we have $(hg)|_v = h|_{g(v)} g|_v = g|_v$.  Hence $hg$ cannot be finitary since $g$ is not finitary.  By hypothesis $a|_{v_n}$ is not finitary, and thus from \eqref{inv} we have that $a^{-1}|_{a(v_n)}$ is not finitary, so that that $a(v_n)$ is the kneading sequence for $a^{-1}$.  Because $a^{-1}a(v_n) = v_n$, multiplication by $a^{-1}$ sends the kneading sequence of $a^{-1}$ to the kneading sequence of $a$.  

For $g \in \langle A \rangle,$ let 
\begin{equation} \label{exp}
g = a_1^{\epsilon_1}a_2^{\epsilon_2} \cdots a_k^{\epsilon_k},
\end{equation}
where the $a_i$ are (not necessarily distinct) elements of $A$, $\epsilon_i \in \{\pm 1\}$, and this expression is minimal length among all words in $A \cup A^{-1}$ giving $g$.  We denote $k$ by $\ell(g)$, and call it the length of $g$.  
From \eqref{prod} it follows that $\ell(g|_v) \leq \ell(g)$ for any $g \in \langle A \rangle$ and any $v \in X^*$.  

Suppose that $g \in \mathcal{N}_0$, so that there exists a non-empty $v \in X^*$ with $g|_v = g$.  If $v_n$ is the length-$n$ initial word of $vv\cdots$, then 
$\ell(g|_{v_n}) = \ell(g)$.    Hence each $a_i$ lies in a cycle of the Moore diagram of $A$, since otherwise at least one would be finitary, implying $\ell(g|_w) < \ell(g)$ for $w$ sufficiently long.  
Let $m$ be the least common multiple of the periods of the (infinite) kneading sequences of the $a_i$. Then 
\begin{multline} \label{vampire}
g|_{v_m} = (a_1^{\epsilon_1}a_2^{\epsilon_2} \cdots a_{k-1}^{\epsilon_{k-1}} a_k^{\epsilon_k})|_{v_m} = \\
a_1^{\epsilon_1}|_{a_2^{\epsilon_2} \cdots a_{k-1}^{\epsilon_{k-1}} a_k^{\epsilon_k}(v_m)} \cdot a_2^{\epsilon_2}|_{a_3^{\epsilon_3} \cdots a_{k-1}^{\epsilon_{k-1}} a_k^{\epsilon_k}(v_m)}\cdots a_{k-1}^{\epsilon_{k-1}}|_{a_k^{\epsilon_k}(v_m)} \cdot a_k^{\epsilon_k}|_{v_m}
\end{multline}
None of the elements in the right-hand side of \eqref{vampire} can be finitary, for otherwise $\ell(g|_w) < \ell(g)$ for sufficiently long $w$.  Hence $v_m$ is the (length-$m$) kneading sequence for $a_k^{\epsilon_k}$, 
$a_k^{\epsilon_k}(v_m)$ is the kneading sequence for $a_{k-1}^{\epsilon_{k-1}}$, 
$a_{k-1}^{\epsilon_{k-1}}a_k^{\epsilon_k}(v_m)$ is the kneading sequence for $a_{k-2}^{\epsilon_{k-2}}$, and so on.  
Thus $g$ determines a path in the kneading graph of $A$, beginning at $v_m$, proceeding to $a_k^{\epsilon_k}(v_m)$, then to $a_{k-1}^{\epsilon_{k-1}}a_k^{\epsilon_k}(v_m)$, and so forth, ending at $a_1^{\epsilon_1} a_2^{\epsilon_2} \cdots a_{k-1}^{\epsilon_{k-1}} a_k^{\epsilon_k}(v_m)$.  There can be no back-tracking because of the minimality of \eqref{exp}.  The path from $v_m$ to $a_k^{\epsilon_k}(v_m)$ follows the arrow labeled $a_k$ if $\epsilon_k = 1$, and runs against the arrow labeled $a_k$ if $\epsilon_k = -1$.  
Assembling the labels along this path from right to left as indicated in the statement of the proposition then yields $g$.  

Conversely, any path in the kneading graph beginning at a vertex $v_m$ yields a word $g = a_1^{\epsilon_1}a_2^{\epsilon_2} \cdots a_k^{\epsilon_k}$.  Now $v_m$ is the (length-$m$) kneading sequence of $a_k^{\epsilon_k}$, and because $v_m$ consists of some number of full cycles of the periodic part of the infinite kneading sequence of $a_k^{\epsilon_k}$, we have $a_k^{\epsilon_k}|_{v_m} = a_k^{\epsilon_k}$.  Similarly, $a_{k-1}^{\epsilon_{k-1}}|_{a_k^{\epsilon_k}(v_m)} = a_{k-1}^{\epsilon_{k-1}}$.  Continuing in this manner we obtain $g|_{v_m} = g$, and hence $g \in \mathcal{N}_0$.

Now take $g \in \mathcal{N}_1$, so that there is some non-empty $v \in X^*$ with $g|_v = v$ and $g(v) = v$.  Then for any length-$n$ initial word $v_n$ of $vv\cdots$, we have $\ell(g|_{v_n}) = \ell(g)$ and $g(v_n) = v_n$.  Hence if $m$ is as above, we have $a_1^{\epsilon_1} a_2^{\epsilon_2} \cdots a_{k-1}^{\epsilon_{k-1}} a_k^{\epsilon_k}(v_m) = v_m$, and so the path in the kneading graph corresponding to $g$ is a circuit.  Conversely, any circuit yields $g$ with $g|_{v_m} = g$ and $g(v_m) = v_m$.  
\end{proof}

\begin{theorem} \label{lengthone}
Let $A$ be a kneading automaton.  Then every cycle in the kneading graph of $A$ has length one.
\end{theorem}

\begin{proof}
Let the cycle in question consist of the vertices $s_0, s_1, \ldots, s_n$, with $s_n = s_0$ and $s_0, \ldots, s_{n-1}$ distinct.  Suppose that $n \geq 2$, so that $s_0 \neq s_1$.  Then assembling the labelings along this path as in Proposition \ref{kneadstab} gives an element $g = a_1^{\epsilon_1}a_2^{\epsilon_2} \cdots a_n^{\epsilon_n}$ with the $a_i$ distinct.  Note that $s_0$ is the (length-$m$) kneading sequence of $a_n^{\epsilon_n}$, and $a_i^{\epsilon_i} \cdots a_n^{\epsilon_n}(s_0) = s_{n-i+1}$ for each $i \geq 1$. In particular, $i = 1$ gives
\begin{equation} \label{acousticjava}
a_1^{\epsilon_1}a_2^{\epsilon_2} \cdots a_n^{\epsilon_n}(s_0) = s_0.
\end{equation}
By Lemma \ref{neklem2}, the set of permutations given by the states of $A$ acting on $X^m$ is tree-like, and hence the cycle diagram of its action is contractible.  If we replace some elements of $A$ by their inverses, the cycle diagram is only altered by changing the directions of some arrows; in particular it is still contractible.  Hence the multi-set of permutations given by $\{a_i^{\epsilon_i} : i = 1, \ldots, n\}$ acting on $X^m$ is a subset of a tree-like multi-set.  From Lemma \ref{treefixed} and \eqref{acousticjava} we then have $a_i^{\epsilon_i}(s_0) = s_0$ for all $i$, which contradicts the fact that $a_n^{\epsilon_n}(s_0) = s_1 \neq s_0$.  
\end{proof}

\begin{theorem} \label{comp}
Let $G = \langle A \rangle$ satisfy all the conditions of Theorem \ref{char}.  Then each component of the kneading graph of $A$ contains at most one cycle.  
\end{theorem}

\begin{proof}
By Proposition \ref{lengthone}, each cycle of the kneading graph of $A$ has length one.  Suppose that there are two such one-cycles at vertices $s_1$ and $s_2$ lying in a connected component of the kneading graph of $A$, and let $a_1$ and $a_2$ be the elements labeling them.  Then $a(s_i) = s_i$ for $i = 1,2$.  Moreover, $s_1$ and $s_2$ have length a multiple of the period of the kneading sequences of $a_1$ and $a_2$, and so $a_i|_{s_i} = a_i$ for $i = 1, 2$.  Now $s_1$ and $s_2$ are in the same component of the kneading graph, and so there is a path connecting them.  Assembling the labelings along this path as in Proposition \ref{kneadstab} gives $h \in G$ with $h(s_1) = s_2$ and $h|_{s_1} = h$, (see the construction in the converse portion of the proof of Proposition \ref{kneadstab}).   
But this contradicts condition (4) of Theorem \ref{char}.  
\end{proof}

\begin{theorem} \label{n1}
Let $G = \langle A \rangle$ satisfy all the conditions of Theorem \ref{char}.  Then every element of $\mathcal{N}_1$ is conjugate to a power of an element of $A \cap \mathcal{N}_1$.  
\end{theorem}

\begin{proof}
By Proposition \ref{kneadstab}, elements of $\mathcal{N}_1$ correspond to circuits in the kneading graph of $A$, which by definition have no back-tracking.  By Theorems \ref{lengthone} and \ref{comp}, each such cycle belongs to a component $C$ having at most a single cycle, which must have length one.  Thus $C$ is either a tree, or becomes a tree when we delete the edge forming the one-cycle.  Every non-trivial circuit in a tree involves back-tracking, and so if $\gamma$ is a non-trivial circuit in $C$, then $C$ must contain a one-cycle at a vertex $s$, labeled by $a \in A$.  Clearly $a \in \mathcal{N}_1$.  Let $s_0$ be the starting point of $\gamma$, and note that if $\gamma$ does not contain the one-cycle at $s$, then it lies entirely within a tree, which is impossible.  Thus $\gamma$ must proceed along the unique path to $s$, go around the one-cycle at $s$ a non-zero number of times in the same direction each time, and return to $s_0$ the same way it came.  If $g$ is the element labeling the path from $s_0$ to $s$ (assembled as in Proposition \ref{kneadstab}), then $g^{-1}$ is the element labeling the reverse path.  Hence the element labeling $\gamma$ is conjugate to a power of $a$.  
\end{proof}

\begin{corollary} \label{n1cor}
Let $G = \langle A \rangle$ satisfy all the conditions of Theorem \ref{char}, and suppose that every element of $A$ that is in a cycle of the reduced Moore diagram either fixes no ends of $X^*$ or fixes infinitely many.  Then $\mathcal{F}(G) = 0$.   
\end{corollary}

\begin{proof}
Condition (1) of Theorem \ref{char} implies that each $a \in A$ is bounded (see Definition \ref{bounded}).  Indeed, for each $a \in A$, the restrictions of $a$ to words of length $n$ consist of the endpoints of all paths of length $n$ in the reduced Moore diagram (following the arrows) starting at $a$.  Because every non-trivial state has a unique incoming arrow, there can be at most one such path ending in each state.  Hence $\#\{v \in X^n : g|_v \neq 1\}$ is bounded by $\#A$.  By Theorem \ref{bndedthm}, $G$ is therefore contracting.  By Lemma \ref{trans}, $G$ contains a spherically transitive element.  We may thus apply Theorem \ref{crystal}, and so to show $\mathcal{F}(G) = 0$ it is enough to show that every $g \in \mathcal{N}_1$ fixes infinitely many ends of $X^*$.   By Theorem \ref{n1}, each $g \in \mathcal{N}_1$ is conjugate to $a^n$ for some $a \in A \cap \mathcal{N}_1$ and $n \in \Z$.  Because $a \in \mathcal{N}_1$, $a$ lies in a cycle of the reduced Moore diagram of $A$ and also fixes at least one end of $X^*$.  Thus by hypothesis $a$ fixes infinitely many ends.  But $a^n$ fixes at least as many elements of each $X^n$ as $a$ does, and hence $a^n$ fixes infinitely many ends of $X^*$.  
\end{proof}


\section{The final steps} \label{sec: last}

Recall that $g \in \Aut(X^*)$ is finitary if all its restrictions at sufficiently long words are the identity.  If $g$ is finitary, then it either fixes no ends of $X^*$ or fixes infinitely many such ends.  Indeed, if $w$ is an end fixed by $g$ and $w_n$ is the length-$n$ initial word of $w$, then $g(w_n) = w_n$ for all $n$ and we may take $n$ large enough so that $g|_{w_n} = 1$.  Thus $g$ fixes all ends with initial word $w_n$, which is an infinite set.  

Throughout this section, when we write $\Fix(g)$ for $g \in \Aut(X^*)$, we mean the set of fixed points of the action of $g$ on $X$ (not on the ends of $X^*$).  

\begin{lemma} \label{finitary}
Let $A$ be a kneading automaton, and let $a, b \in A$ be finitary with $a \neq b$.   Then at least one of $a, b$ fixes infinitely many ends of $X^*$.   
\end{lemma}

\begin{proof}
By 
part (1) of Lemma \ref{treeperms}, $\#\Fix(a) + \#\Fix(b) \geq 2$, and hence there exist elements $x_0, y_0 \in X$ that are fixed by either $a$ or $b$.  Renaming if necessary, assume that $a$ fixes $x_0$, and let $a_1 = a|_{x_0}$.  Let $b_1 = a|_{y_0}$ if $a(y_0) = y_0$ and $b_1 = b|_{y_0}$ if $b(y_0) = y_0$.  Applying part (1) of Lemma \ref{treeperms} again, there exist 
$x_1, y_1 \in X$ that are fixed by either $a_1$ or $b_1$.  Renaming again if necessary, assume $a_1(x_1) = x_1$,  and let $a_2 = a_1|_{x_1}$.  Let $b_2 = a_1|_{y_1}$ if $a_1(y_1) = y_1$ and $b_2 = b_1|_{y_1}$ if $b_1(y_1) = y_1$.  Proceeding in this manner yields  
a sequence $a_1, a_2, \ldots$ of elements of $A$ and words $w_n = x_0x_1\cdots x_n$ such that $a(w_n) = w_n$ and $a|_{w_n} = a_{n+1}$.  

Because $a$ is finitary, there exists $j > 0$ such that all $a_n$ are trivial for $n \geq j$.  Hence $a$ fixes the word $w_j = x_0x_1 \cdots x_j$ and $a|_{w_j} = 1$, implying that $a$ fixes all ends of $X^*$ with initial word $w_j$, which is an infinite set.
\end{proof}


\begin{theorem} \label{nolongcycle}
Let $G = \langle A \rangle$ satisfy the conditions of Theorem \ref{char}, and suppose $A$ acts on a set $X$ with $\#X = d$.  Let $C$ be a cycle of the reduced Moore diagram of $A$.  If $C$ contains at least two elements and one of them fixes a non-empty, finite set of ends of $X^*$, then $d$ is odd, $C$ is the only cycle, and up to conjugation in $\Aut(X^*)$ we have $A = \{1, a,b\}$ with 
\begin{equation} \label{howitis}
a = \sigma(b, 1, 1,\ldots, 1) \qquad b = \tau(1, 1, \ldots, 1, a),
\end{equation}
where $\sigma$ and $\tau$ are products of disjoint transpositions, $\sigma$ fixes only $0 \in X$, and $\tau$ fixes only $d-1 \in X$.  In particular, $\langle A \rangle$ is infinite dihedral.   
\end{theorem}

\begin{proof}
Let $C = \{c_1, c_2, \ldots, c_n\}$, so that there is an arrow in the reduced Moore diagram from $c_i$ to $c_{i+1}$ for $i = 1, \ldots, n-1$ and also from $c_n$ to $c_1$.  Let $v_1 = x_1x_2 \cdots x_n$ be the length-$n$ kneading sequence of $c_1$, implying that $v_i = x_ix_{i+1} \cdots x_nx_1 \cdots x_{i-1}$ is the kneading sequence of $x_i$.  Then $c_i|_{v_i} = c_i$ and from \eqref{lifts},
$$c_i(v_i) = c_i(x_i)c_{i+1}(x_{i+1}) \cdots c_n(x_n)c_1(x_1) \cdots c_{i-1}(x_{i-1}).$$
Hence $c_i$ fixes $v_i$ if and only if each $c_j$ fixes $x_j$.  

Assume that $C$ contains at least two elements and one of them fixes a non-empty, finite set of ends of $X^*$.  Suppose first that $c_i$ does not fix $v_i$ for some $i$.
If $c_i$ fixes an end $w$ of $X^*$, then $w$ cannot be the (infinite) kneading sequence of $c_i$.  Letting $w_j$ be the length-$j$ initial word of $w$, we can thus take $j$ large enough so that $c_i|_{w_j}$ is finitary.  Let $b = c_i|_{w_j}$, and note that for $k > j$, $c_i|_{w_k}$ is a restriction of $b$ at a word of length $k-j$.  Because $b$ is finitary, we may take $k$ large enough so that $c_i|_{w_k} = 1$.  But $c_i$ fixes $w$, and thus $c_i(w_k) = w_k$, ensuring that $c_i$ fixes all ends of $X^*$ with initial word $w_k$.  
Hence all $c_i$ either fix no ends of $X^*$ or infinitely many, a contradiction.

Suppose now that $c_i$ fixes $v_i$ for some $i$ (equivalently, all $i$).  Then for each $j \neq i$ there is some word $w_j$ with $c_i(w_j) = w_j$ and $c_i|_{w_j} = c_j$.  Thus if any $c_j$ fixes infinitely many ends of $X^*$, then the same conclusion holds for all the $c_j$.  If $n \geq 3$, then we claim
\begin{equation} \label{lcdtrio}
\#\Fix(c_1) + \#\Fix(c_2) + \#\Fix(c_3) \geq 5.
\end{equation}
To see why, note that $\#\Fix(c_i) \geq 1$ for all $i$ by assumption, and if $\#\Fix(c_i) + \#\Fix(c_j) \leq 3$ for some $i \neq j$, then part (2) of Lemma \ref{treeperms} gives $\#\Fix(c_k) = d$ for the remaining $k$.  If $d = 2$, then we must have $\#\Fix(c_i) = 2$ for all $i$, and \eqref{lcdtrio} holds.  If $d \geq 3$ and $\#\Fix(c_i) = \#\Fix(c_j) = 1$, then $\#\Fix(c_k) = d$ and again \eqref{lcdtrio} holds.  

Thus by \eqref{lcdtrio} there are $y_1 \neq y_2 \in X \setminus \{x_1, x_2, x_3\}$ with $c_i(y_1) = y_1$ and $c_j(y_2) = y_2$ (here $i$ is allowed to equal $j$).  Now $c_i|_{y_1}$ and $c_i|_{y_1}$ are distinct and finitary, and by Lemma \ref{finitary} at least one of them fixes infinitely many ends of $X^*$.  Thus at least one of the $c_i$ fixes infinitely many ends of $X^*$, and hence all do.  This gives a contradiction.

We have therefore shown that $\#C = 2$, and we write $C = \{a, b\}$.  Let $x_1x_2$ be the kneading sequence of $a$, implying that $x_2x_1$ is the kneading sequence of $b$, and recall that both $a$ and $b$ must fix their kneading sequences.  If $\#\Fix(a) + \#\Fix(b) \geq 4$, then there exist
$y_1, y_2 \in X \setminus \{x_1, x_2\}$ fixed by either $a$ or $b$.  As in the previous paragraph, we conclude that both $a$ and $b$ fix infinitely many ends of $X^*$, a contradiction.  Hence $\#\Fix(a) + \#\Fix(b) \leq 3$, and from part (2) of Lemma \ref{treeperms} we have that every $a' \in A \setminus \{a, b\}$ must act as the identity on $X$.  If $a'$ is in the component of the reduced Moore diagram containing $C$, then it cannot be part of $C$, and neither can any of its restrictions.  Thus $a'$ acts trivially on $X^*$.  If $a'$ is in a component of the reduced Moore diagram of $A$ that does not contain $C$, then every element of this component must act trivially on $X$.  Since components are closed under restriction, it follows that every element of this component is trivial.  Thus 
$A = \{1, a, b\}$.  

Now if $\#\Fix(a) + \#\Fix(b) = 3$, then there is some $y \in X \setminus \{x_1, x_2\}$ that we may assume without loss is fixed by $a$.  Then $a|_y$ is not in the cycle $\{a, b\}$, and so $a|_y = 1$, showing that $a$ fixes infinitely many ends of $X^*$.  Thus $b$ must as well, which is a contradiction.  Therefore 
$\#\Fix(a) = \#\Fix(b) = 1$, and we must have $\Fix(a) \cap \Fix(b) = \emptyset$.  Otherwise $a$ and $b$ have the same kneading sequence, and thus give two one-cycles in the same component of the kneading graph of $A$, violating Theorem \ref{comp}.  Hence conjugating by an appropriate $\gamma(1, 1, \ldots, 1) \in \Aut(X^*)$ we may assume that 
$\Fix(a) = \{0\}$ and $\Fix(b) = \{d-1\}$, giving the forms in \eqref{howitis}.  Moreover, because $\#\Fix(a) + \#\Fix(b) = 2$ we must have equality in \eqref{euler}, which implies that the reduced cycle diagram of $\{\sigma, \tau\}$ has $d-1$ 2-cells.  By Lemma \ref{treecyc} it follows that $\sigma$ and $\tau$ are products of disjoint transpositions, and hence $d$ must be odd.  

Now $a^2 = (a^2, 1, 1, \ldots, 1)$, and so $a^2 = 1$, and similarly $b^2 = 1$.  By Lemma \ref{trans} or Proposition \ref{sphertrans} we have that $ab$ has infinite order, and conjugating $ab$ by $a$ gives $ba$, which is $(ab)^{-1}$.  Hence $\langle A \rangle$ is infinite dihedral.  
\end{proof}

\begin{theorem} \label{titus}
Let $G = \langle A \rangle$ satisfy the conditions of Theorem \ref{char}, where $A$ acts on a set $X$ with $\#X = d$.  Let $C = \{c\}$ be a 1-cycle in the reduced Moore diagram of $A$, and suppose that $c$ fixes at least two points of $X$.  If $c$ fixes a non-empty, finite set of ends of $X^*$, then $d$ is even, $C$ is the only cycle, and up to conjugation in $\Aut(X^*)$ we have $A = \{c, a, 1\}$ with
\begin{equation} \label{howitis2}
c = \sigma(c, a, 1,1, \ldots, 1) \qquad a = \tau(1, 1, \ldots, 1),
\end{equation}
where $\sigma$ and $\tau$ are products of disjoint transpositions, $\Fix(\sigma) = \{0, 1\}$, and 
$\Fix(\tau) = \emptyset$.  In particular, $\langle A \rangle$ is infinite dihedral.   
\end{theorem}
 
\begin{proof}
Suppose that $c$ fixes a non-empty, finite set of ends of $X^*$.  Let $x \in X$ be such that $c|_x = x$.  If $\#\Fix(c) \geq 3$, then there are $y_1, y_2 \in X \setminus \{x\}$ with $y_1 \neq y_2$ and $c(y_i) = y_i$.  Then $c|_{y_i}$ is finitary for $i = 1, 2$, and by Lemma \ref{finitary} one of them fixes infinitely many ends of $X^*$.  It follows that $c$ fixes infinitely many ends of $X^*$, a contradiction.  

Hence $\#\Fix(c) = 2$.  If $c(x) \neq x$, then $\Fix(c) = \{y_1, y_2\} \subseteq X \setminus \{x\}$, and as before we have a contradiction.   Thus $\Fix(c) = \{x, y\}$ for some $y \neq x$.  Let $a = c|_y$.  
If $\#\Fix(a) \geq 2$, then there are $z_1, z_2 \in X \setminus \{y\}$ with $z_1 \neq z_2$ and $a(y_i) = y_i$.  But $c|_{z_i}$ is finitary for $i = 1, 2$, and as before this gives a contradiction.  Hence $\#\Fix(a) \leq 1$, so that $\#\Fix(c) + \#\Fix(a) \leq 3$, and we may apply part (2) of Theorem \ref{treeperms} as in the proof of Theorem \ref{nolongcycle} to get that $A = \{c, a, 1\}$.  If $\#\Fix(a) = 1$, then the restriction of $a$ at this fixed point must be trivial, which implies that $c$ fixes infinitely many ends of $X^*$.  Thus 
$\Fix(a) = \emptyset$.  Conjugating by an appropriate element $\gamma(1, 1, \ldots, 1) \in \Aut(X^*)$ allows us to move $x$ to $0$ and $y$ to $1$, so that $c$ and $a$ have the forms in \eqref{howitis2}.  
Because $\#\Fix(c) + \#\Fix(a) = 2$ we must have equality in \eqref{euler}, which implies that the reduced cycle diagram of $\{\sigma, \tau\}$ has $d-1$ 2-cells.  By Lemma \ref{treecyc} it follows that $\sigma$ and $\tau$ are products of disjoint transpositions, and hence $d$ must be even.  That $\langle A \rangle$ is dihedral follows just as in the proof of Theorem \ref{nolongcycle}.
\end{proof} 

We now have all the tools in place to finish proving our main result. 

\begin{proof}[Proof of Theorem \ref{monodromy}]  Recall that $f \in \C[z]$ is exceptional if there exists a finite, non-empty set $\Sigma$ such that $f^{-1}(\Sigma) \setminus \critf = \Sigma$. Suppose $f$ is non-exceptional, put $d = \deg f$, and let $G \leq \Aut(X^*)$ be a standard action of $\img(f)$ on $X^*$. There is an automaton $A$ that generates $G$ and satisfies the conditions of Theorem \ref{char}. 

Suppose that there is a cycle $C$ in the reduced Moore diagram of $A$ such that some element of $C$ fixes a non-empty, finite set of ends of $X^*$.  If $\#C \geq 2$, then by Theorem \ref{nolongcycle}, $G$ is generated by two elements of the form \eqref{howitis}.  From Theorem \ref{compmon} it follows that the post-critical set of $f$ consists of $\{z_0, z_1\} \subset \C$ where $a = g_{z_0}, b = g_{z_1}$, $z_0 \neq z_1$, and neither $z_0$ nor $z_1$ is critical.  Moreover, $f^{-1}(\{z_0, z_1\}) \setminus \{z_0, z_1\}$ is contained in the set of critical points of $f$, and hence $f$ is exceptional (indeed, by Proposition \ref{chebclass} it is conjugate to $-T_d$ for some odd $d$). If $C = \{c\}$, then let $x \in X$ be such that $c|_x = c$, and let $\Fix(c)$ be the set of fixed points in the action of $c$ on $X$.  By assumption $\Fix(c)$ is non-empty. 
If $\#\Fix(c) \geq 2$, then by Theorem \ref{titus}, $G$ is generated by two elements of the form \eqref{howitis2}. As before Theorem \ref{compmon} implies that $f$ is exceptional (it is conjugate to $T_d$ for some even $d$).  If $\Fix(c) = \{y\}$ with $y \neq x$, then $c|_y$ must be finitary and hence fix either no ends or infinitely many ends by the remarks at the beginning of this section. Thus $c$ either fixes no ends or infinitely many ends, contrary to our assumption. If $\Fix(c) = \{x\}$, then by Theorem \ref{compmon} we have $x = g_{z_0}$ where $z_0$ is non-critical and $f^{-1}(\{z_0\}) \setminus \{z_0\}$ is contained in the set of critical points of $f$.  Once again $f$ must be exceptional (it is conjugate to a polynomial of the form \eqref{oneptexcept}). 

Therefore the reduced Moore diagram of $A$ contains no such cycle $C$. By Corollary \ref{n1cor} we have $\mathcal{F}(G) = 0$.
\end{proof}

In order to prove Proposition \ref{exceptprop} and the remarks preceding it, we make a brief study of exceptional polynomials (see the discussion following Lemma 2.3 in \cite{lkr} for similar remarks in the case of rational functions). \label{exceptdisc} Suppose that $f$ has degree $d$ and is exceptional, with $\Sigma$ a finite set such that $f^{-1}(\Sigma) \setminus \critf = \Sigma$. Following \cite{lkr}, we note that each preimage of a point in $\Sigma$ is either a critical point or in $\Sigma$. Hence, letting $\Gamma = \critf \cap f^{-1}(\Sigma)$, we have
\begin{equation*} 
d \cdot \#\Sigma - \sum_{c \in \Gamma} (\deg(c) - 1) = \#f^{-1}(\Sigma) \leq \#\Sigma + \#\Gamma,
\end{equation*}
where $\deg(c)$ is the local degree of $f$ at $c$. Because $f$ has $d-1$ critical points up to multiplicity, this gives
\begin{equation*} 
(d-1)\#\Sigma = \sum_{c \in \Gamma} (\deg(c) - 1) + \#\Gamma \leq 2(d-1),
\end{equation*}
with equality holding if and only if $f$ has $d-1$ distinct critical points (necessarily each having multiplicity 2), all of which are contained in $f^{-1}(\Sigma)$. 


We now give a complete characterization of the case $\#\Sigma = 2$. While our result is not new, the statement and proof differ in form from the standard treatments in the literature (e.g. \cite[Theorem 19.9]{milnor}).  Recall that the Chebyshev polynomial of degree $d$ is given by $T_d(z) = \cos (d \arccos z)$.  Its critical set $\Gamma$ consists of $d-1$ distinct points and satisfies $\Gamma \cap \{\pm 1\} = \emptyset$, $T_d(\Gamma) = \{\pm 1\}$, $T_d(1) = 1$, $T_d(-1) = (-1)^d$, and $T_d^{-1}\{\pm 1\} = \{\pm 1\} \cup \Gamma$.  

\begin{proposition} \label{chebclass}
Let $f \in \C[z]$ be exceptional, with $\Sigma = \{z_0, z_1\}$ and $z_0 \neq z_1$. Then one of the following holds:
\begin{enumerate}
\item $f(z_0) = z_0$, $f(z_1) = z_1$, and $f$ is conjugate to $T_d$ for some odd $d$. 
\item $f(z_0) = z_1$, $f(z_1) = z_0$, and $f$ is conjugate to $-T_d$ for some odd $d$.
\item $f(z_0) = f(z_1) \in \{z_0, z_1\}$, and $f$ is conjugate to $T_d$ for some even $d$. 
\end{enumerate}
\end{proposition}

\begin{proof}
We outline an algebraic approach, which differs somewhat from the well-known geometric arguments (see e.g. \cite[Theorem 19.9]{milnor}). First note that by the definition of exceptional polynomial, $f(\Sigma) \subseteq \Sigma$ and $\Sigma$ contains no critical points. By the discussion preceding the Proposition, $f$ has $d-1$ critical points, all of which have multiplicity 2. Applying an appropriate affine conjugation, we may assume that $z_0  = -1$ and $z_1 = 1$. 
In case (1) of the proposition, we then have
\begin{align} \label{eqone}
f(z) + 1 = r(z+1)(g(z))^2 \qquad f(z) - 1 = r(z - 1)(h(z))^2,
\end{align}
where $r \in \C$, and $g, h \in \C[z]$ are monic and relatively prime. 
Differentiating gives $f' = r[2(z+1)g'g + g^2] = r[2(z-1)h'h + h^2]$, but the roots of $f'$ are the same as the roots of $gh$, whence $f' = drgh$, where $d$ is the degree of $f$. Because $g$ and $h$ are relatively prime, we obtain $dh = 2(z+1)g' + g$ and $dg = 2(z-1)h' + h$. Differentiating again and substituting yields 
$$d^2g 
= 4(z+1)(z-1)g'' + (8z - 4)g' + g.$$
This differential equation gives a recurrence relation on the coefficients of $g$; with the assumption that $g$ is monic, this uniquely determines all coefficients of $g$.  Because $f(1) = 1$, we have by \eqref{eqone} that $2 = 2r(g(1))^2$, thereby determining $r$, and thus also $f$. However, $T_d$ clearly satisfies the same conditions as $f$, and thus $f = T_d$. Part (1) of the proposition follows. The other parts proceed similarly. 
\end{proof}

To prove Proposition \ref{exceptprop}, we require the following result. 

\begin{proposition} \label{chebcomp}
Let $\#X = d$ and suppose that $G \leq \Aut(X^*)$ is generated by $\{a, b\}$ where $a$ and $b$ are distinct and non-trivial, $a^2 = b^2 = 1$, and $ab$ is spherically transitive.  
Then $\mathcal{F}(G) = r/4$, where $r$ is the number of elements in $\{a, b\}$ fixing at least one end of $X^*$.    
\end{proposition}
\begin{proof}
Because $ab$ is spherically transitive, we have that for any $n \geq 1$ its action on $X^n$ is a $d^n$-cycle.  Clearly conjugation of $ab$ by $a$ gives $(ab)^{-1}$, and it follows that if $G_n$ is the action of $G$ on $X^n$ then $G_n$ is dihedral of order $2d^n$ and a complete list of its elements is given by the actions of $(ab)^k$ and $a(ab)^k$ for $k = 0, \ldots, d^n$.  None of the elements of the form 
$(ab)^k$ for $k = 1, \ldots, d^n - 1$ can have fixed points in $X^n$.  An element of the form $a(ab)^k$ is conjugate to $b$ if $k$ is odd and to $a$ if $k$ is even.  Now the action of $a$ either has a fixed point in $X^n$ for all $n$ (if $a$ fixes an end of $X^*$) or has no fixed points in $X^n$ for $n$ large enough, and similar statements hold for $b$.  Thus for $n$ sufficiently large, the number of elements of $G_n$ fixing at least one point of $X^n$ is 
$1 + r(d^n/2)$.  Dividing by $\#G_n = 2d^n$ and letting $n \to \infty$ gives $\mathcal{F}(G) = r/4$.  
\end{proof}

\begin{proof}[Proof of Proposition \ref{exceptprop}]
Let $f(z)$ have degree $d$, and let $G \leq \Aut(X^*)$ be a standard action of $\img(f)$ on $X^*$. If $f(z)$ is conjugate to $T_d$ for $d$ even, then it follows from Theorem \ref{compmon} that $G$ is generated by two elements of the form \eqref{howitis2}.
Proposition \ref{chebcomp} then applies to show $\mathcal{F}(G) = 1/4$. 
If $f(z)$ is conjugate to $-T_d$ for $d$ odd, then it follows from Theorem \ref{compmon} that $G$ is generated by two elements of the form \eqref{howitis}. If $f(z)$ is conjugate to $T_d$ for $d$ odd, then $G$ is generated by two elements of the form 
$$a = \sigma(a, 1, 1,\ldots, 1) \qquad b = \tau(1, 1, \ldots, 1, b),$$
with assumptions as in \eqref{howitis}. In either case Proposition \ref{chebcomp} applies to show $\mathcal{F}(G) = 1/2$.

\end{proof}

\section*{Acknowledgements}
I would like to thank Lasse Rempe, Juan Rivera-Letelier, and Mikhail Lyubich for helpful comments and references to the literature on exceptional maps. I also extend my thanks to the Institute for Computational and Experimental Research in Mathematics, where I presented and received valuable feedback on some of these results as part of the semester on complex and arithmetic dynamics.


\begin{thebibliography}{10}

\bibitem{bartnek}
Laurent Bartholdi and Volodymyr~V. Nekrashevych.
\newblock Iterated monodromy groups of quadratic polynomials. {I}.
\newblock {\em Groups Geom. Dyn.}, 2(3):309--336, 2008.

\bibitem{bartvir}
Laurent Bartholdi and B{\'a}lint Vir{\'a}g.
\newblock Amenability via random walks.
\newblock {\em Duke Math. J.}, 130(1):39--56, 2005.

\bibitem{bondnek}
E.~Bondarenko and V.~Nekrashevych.
\newblock Post-critically finite self-similar groups.
\newblock {\em Algebra Discrete Math.}, (4):21--32, 2003.

\bibitem{forster}
Otto Forster.
\newblock {\em Lectures on {R}iemann surfaces}, volume~81 of {\em Graduate
  Texts in Mathematics}.
\newblock Springer-Verlag, New York, 1991.
\newblock Translated from the 1977 German original by Bruce Gilligan, Reprint
  of the 1981 English translation.

\bibitem{grigorchuk}
Rostislav Grigorchuk, Dmytro Savchuk, and Zoran Sunic.
\newblock The spectral problem, substitutions and iterated monodromy, 2007.

\bibitem{grimmett}
Geoffrey~R. Grimmett and David~R. Stirzaker.
\newblock {\em Probability and random processes}.
\newblock Oxford University Press, New York, third edition, 2001.

\bibitem{galmart}
Rafe Jones.
\newblock Iterated {G}alois towers, their associated martingales, and the
  {$p$}-adic {M}andelbrot set.
\newblock {\em Compos. Math.}, 143(5):1108--1126, 2007.

\bibitem{quaddiv}
Rafe Jones.
\newblock The density of prime divisors in the arithmetic dynamics of quadratic
  polynomials.
\newblock {\em J. Lond. Math. Soc. (2)}, 78(2):523--544, 2008.

\bibitem{lkr}
Jeremy Kahn, Mikhail Lyubich, and Lasse Rempe.
\newblock A note on hyperbolic leaves and wild laminations of rational
  functions.
\newblock {\em J. Difference Equ. Appl.}, 16(5-6):655--665, 2010.

\bibitem{lyubich}
Mikhail Lyubich and Yair Minsky.
\newblock Laminations in holomorphic dynamics.
\newblock {\em J. Differential Geom.}, 47(1):17--94, 1997.

\bibitem{earlysmirnov}
N.~Makarov and S.~Smirnov.
\newblock Phase transition in subhyperbolic {J}ulia sets.
\newblock {\em Ergodic Theory Dynam. Systems}, 16(1):125--157, 1996.

\bibitem{smirnov}
N.~Makarov and S.~Smirnov.
\newblock On ``thermodynamics'' of rational maps. {I}. {N}egative spectrum.
\newblock {\em Comm. Math. Phys.}, 211(3):705--743, 2000.

\bibitem{milnor}
John Milnor.
\newblock {\em Dynamics in one complex variable}, volume 160 of {\em Annals of
  Mathematics Studies}.
\newblock Princeton University Press, Princeton, NJ, third edition, 2006.

\bibitem{nek}
Volodymyr Nekrashevych.
\newblock {\em Self-similar groups}, volume 117 of {\em Mathematical Surveys
  and Monographs}.
\newblock American Mathematical Society, Providence, RI, 2005.

\bibitem{rosen}
Michael Rosen.
\newblock {\em Number theory in function fields}, volume 210 of {\em Graduate
  Texts in Mathematics}.
\newblock Springer-Verlag, New York, 2002.

\end{thebibliography}
\end{document}